\documentclass[11pt]{article}

\usepackage[a4paper,margin=1in]{geometry}
\usepackage{amsmath,amssymb,amsthm,mathtools}
\usepackage{graphicx}
\usepackage{flafter}
\usepackage{float}
\usepackage{booktabs}
\usepackage{enumitem}
\usepackage{hyperref}
\usepackage{microtype}
\usepackage{authblk}

\hypersetup{
 colorlinks=true,
 linkcolor=black,
 citecolor=black,
 urlcolor=black
}

\newtheorem{assumption}{Assumption}
\newtheorem{theorem}{Theorem}

\newtheorem{corollary}{Corollary}
\newtheorem{remark}{Remark}

\newcommand{\R}{\mathbb{R}}
\newcommand{\dd}{\,\mathrm{d}}

\newcommand{\D}{\mathrm{D}}
\newcommand{\abs}[1]{\left\lvert #1\right\rvert}
\newcommand{\norm}[1]{\left\lVert #1\right\rVert}
\newcommand{\sdet}[1]{\sqrt{\det\left(#1\right)}}
\newcommand{\given}{\,|\,}
\newcommand{\detcorr}{\texttt{detcorr}}

\title{Constraint residuals, graph posteriors, and determinant-corrected full-space targets in Bayesian inverse problems}

\author{Jonathon Cottom}
\author{Emilia Olsson\thanks{Email: k.i.e.olsson@uva.nl}}

\affil{Institute for Theoretical Physics, University of Amsterdam, Science Park 904, 1098 XH, Amsterdam, the Netherlands}
\affil{Advanced Research Center for Nanolithography, Science Park 106, 1098 XG Amsterdam, The Netherlands}

\date{\today}

\begin{document}

\maketitle

\begin{abstract}
Bayesian inverse problems constrained by state equations are often sampled in a full parameter-state space by penalising the residual, rather than in a reduced space where the state is eliminated. We show that these formulations are not automatically equivalent as posterior measures. For finite-dimensional discretisations of equality-constrained inverse problems, assume the state equation \(c(\theta,u)=0\) has a unique solution \(u=G(\theta)\) and nonsingular state Jacobian \(\D_u c\). The reduced posterior, its graph lift, and the zero-noise residual posterior are then distinct. A local change of variables shows that an uncorrected Gaussian residual penalty converges, after marginalisation over \(u\), to the reduced density multiplied by \(\abs{\det \D_u c(\theta,G(\theta))}^{-1}\). Thus algebraically equivalent residuals can define the same feasible set but different limiting posteriors. We derive determinant corrections for unweighted, weighted, and rescaled residual penalties that have the graph-lifted reduced posterior as their hard-constraint limit. The result separates feasibility from posterior calibration: driving the residual to zero is not sufficient for exact sampling of the graph-lifted reduced posterior unless the sampling or correction step targets the corresponding corrected density.
\end{abstract}

\noindent\textbf{Keywords:} Bayesian inverse problems; constrained posterior; augmented Lagrangian; ADMM; coarea formula; manifold MCMC; sequential Monte Carlo; Stein variational gradient descent.


\section{Introduction}

Many inverse problems in the physical sciences are constrained by a model equation. After discretisation, a typical problem has an unknown parameter \(\theta\in\Theta\subset\R^p\), a state variable \(u\in U\subset\R^q\), observations \(y\), and a state equation
\begin{equation}
 c(\theta,u)=0. \label{eq:constraint_intro}
\end{equation}
Examples include elliptic and parabolic coefficient inverse problems, electromagnetic and acoustic inverse scattering, full waveform inversion, inverse transport, and kinetic or rate-equation models \cite{Tarantola1984,Tarantola2005,VirieuxOperto2009}. In the standard reduced Bayesian formulation, one solves \eqref{eq:constraint_intro} for \(u=G(\theta)\) and samples
\begin{equation}
 \pi_{\rm red}(\theta\given y)
 \propto
 \pi_0(\theta) L(y\given \theta,G(\theta)), \label{eq:reduced_intro}
\end{equation}
where \(\pi_0\) is the prior density and \(L\) is the likelihood. This reduced formulation is conceptually clean and aligns with the classical Bayesian inverse-problem framework \cite{KaipioSomersalo2005,Stuart2010,CalvettiKaipioSomersalo2014,DashtiStuart2017}, but it can be difficult to sample when each likelihood evaluation requires a full PDE solve or when the inverse problem is highly ill conditioned \cite{BuiThanhGhattasMartinStadler2013, FlathWilcoxAkcelikHillVanBloemenWaandersGhattas2011, MartinWilcoxBursteddeGhattas2012,PetraMartinStadlerGhattas2014, CuiMartinMarzoukSolonenSpantini2014,CuiLawMarzouk2016}.

Full-space and dual-space formulations introduce \(u\) as an inference variable and enforce \eqref{eq:constraint_intro} by a penalty, an augmented Lagrangian, a splitting scheme, or a constrained dynamics. These methods are natural from the point of view of deterministic all-at-once and PDE-constrained optimisation \cite{HaberAscher2001,HinzePinnauUlbrichUlbrich2009,BorziSchulz2012,BenziGolubLiesen2005,vanLeeuwenHerrmann2016}, and they can improve conditioning by relaxing the state equation while the algorithm is far from a feasible solution. They are also increasingly used as building blocks for uncertainty quantification: for example, function-space, Riemannian, and manifold MCMC methods exploit PDE-induced geometry \cite{BuiThanhGirolami2014,CotterRobertsStuartWhite2013, BeskosPinskiSanzSernaStuart2011,HairerStuartVollmer2014,Law2014, BuiThanhNguyen2016,GirolamiCalderhead2011, BeskosGirolamiLanFarrellStuart2017,DiaconisHolmesShahshahani2013, ByrneGirolami2013,ZappaHolmesCerfonGoodman2018, LelievreRoussetStoltz2019,GrahamThieryBeskos2022}, ADMM-type splitting has been developed for MCMC \cite{EcksteinBertsekas1992,BoydParikhChuPeleatoEckstein2011,ParikhBoyd2014,VonoPaulinDoucet2022}, and particle variational methods such as Stein variational gradient descent (SVGD) have been coupled with constrained or augmented-Lagrangian updates \cite{LiuWang2016,JiaLiMeng2022, SiahkoohiAghazadeGholami2026,AghazadeSiahkoohiGholami2026}.

This paper concerns a target-measure issue that is independent of any single algorithm. Driving the residual \(c(\theta,u)\) to zero is necessary for a hard-constraint limit, but it is not by itself enough to determine a Bayesian posterior measure. The feasible set
\[
 \Gamma = \{(\theta,u): c(\theta,u)=0\}
\]
is unchanged if one replaces \(c\) by \(A(\theta)c\), where \(A(\theta)\) is any invertible matrix. Such residual transformations are optimisation-equivalent: they define the same feasible points. They are not automatically posterior-equivalent, because a residual penalty also specifies a reference volume in the residual coordinates. A full-space Bayesian formulation must therefore answer a modelling question before it answers an algorithmic one: should the target be the reduced posterior lifted to the graph \(u=G(\theta)\), the zero-noise posterior for an observed residual, or another constrained measure?

The paper distinguishes three objects that are often conflated. The first is the reduced posterior on \(\theta\), obtained by solving the state equation and evaluating the likelihood at \(u=G(\theta)\). The second is the graph-lifted reduced posterior on \(\Gamma\), which is the push-forward of the reduced posterior under \(\theta\mapsto(\theta,G(\theta))\). The third is the zero-noise residual posterior obtained by placing a small-noise likelihood on the residual coordinates \(c(\theta,u)\) in the ambient full space. These measures live on the same feasible set in the hard-constraint limit, but they need not induce the same distribution over \(\theta\).

The determinant factor derived below is not presented as a new coarea formula. The coarea formula, disintegration, and finite-dimensional change-of-variables identities are classical. The contribution is to show that this classical volume factor changes the limiting posterior targeted by common full-space residual-penalty formulations in Bayesian inverse problems, and to turn that observation into corrected targets, weighted-residual formulas, and reproducible diagnostics.

We do not argue against augmented-Lagrangian, splitting, ADMM, SMC, MCMC, or particle variational methods. On the contrary, these methods are useful once the target measure is specified. They can provide proposal geometry, preconditioning, tempering paths, initialisation strategies, and variational approximations. Posterior correctness, however, comes from an invariant transition, a valid sequential correction, or a clearly declared variational approximation to a specified target. Residual convergence and posterior calibration must therefore be monitored separately.

The main result is simple. Suppose that \(c(\theta,u)=0\) has a unique solution \(u=G(\theta)\) and that \(\D_u c\) is nonsingular. Let
\begin{equation}
 r(\theta,u) = \pi_0(\theta)L(y\given\theta,u). \label{eq:r_def}
\end{equation}
Then the \(\theta\)-marginal of the naive penalty posterior
\begin{equation}
 \pi_\rho(\theta,u)
 \propto
 r(\theta,u)
 \exp\!\left[-\frac{\rho}{2}\norm{c(\theta,u)}^2\right]
 \label{eq:naive_penalty_intro}
\end{equation}
converges, as \(\rho\to\infty\), to a density proportional to
\begin{equation}
 r(\theta,G(\theta))\,
 \abs{\det\D_u c(\theta,G(\theta))}^{-1}. \label{eq:naive_limit_intro}
\end{equation}
This is generally not the reduced posterior. To recover the graph-lifted reduced posterior from a full-space penalty, one should instead use
\begin{equation}
 \widetilde{\pi}_\rho(\theta,u)
 \propto
 r(\theta,u)
 \abs{\det\D_u c(\theta,u)}
 \exp\!\left[-\frac{\rho}{2}\norm{c(\theta,u)}^2\right].
 \label{eq:corrected_penalty_intro}
\end{equation}
The determinant correction in \eqref{eq:corrected_penalty_intro} is the finite-dimensional Jacobian needed to make the full-space residual penalty converge to the same parameter posterior as the reduced formulation.

The constructive correction is accompanied by a sampler-agnostic software package, \detcorr, which evaluates the determinant term, assembles corrected finite-penalty log densities, and emits diagnostics that distinguish constraint feasibility from posterior calibration. The package is not needed for the mathematical statements below, but it makes the residual-scaling tests and the benchmark in Section~\ref{sec:pde_protocol} directly reproducible.

\subsection*{Contributions}

The paper makes four contributions.
\begin{enumerate}[leftmargin=*]
 \item We identify a target-measure ambiguity in full-space Bayesian inverse-problem penalties by distinguishing the reduced posterior, its graph lift to the feasible manifold, and the zero-noise residual posterior induced by a penalty on \(c\).
 \item We prove a finite-dimensional penalty-limit theorem, under explicit residual-coordinate, tail, domination, and normalisation assumptions, showing how the naive penalty posterior differs from the graph-lifted reduced posterior by a state-Jacobian volume factor.
  \item We derive corrected full-space targets for the graph-lifted reduced posterior, including the weighted-residual correction, and distinguish hard-constraint invariance from exact finite-\(\rho\) invariance under consistently transformed residual weights.
 \item We give sampler-agnostic SMC, MCMC, manifold, and particle-method templates in which augmented-Lagrangian and ADMM geometry is used as proposal, preconditioning, tempering, initialisation, or variational machinery, and we validate the residual-scaling invariance and elliptic benchmark through code-generated diagnostics.
\end{enumerate}

The results are stated for finite-dimensional discretisations. This is the natural level at which many computational algorithms are implemented. Infinite-dimensional versions require an additional choice of reference measure and a careful treatment of determinants, discretisation limits, and possible singularity of the state map; these issues are discussed in Section~\ref{sec:discussion}.

\section{Reduced, graph-lifted, and residual posteriors}
\label{sec:posteriors}

Let \(\Theta\subset\R^p\) and \(U\subset\R^q\) be open sets. We assume that \(\theta\) is the inverse parameter and \(u\) is the state. The data likelihood is represented by a nonnegative density \(L(y\given\theta,u)\), and the prior density of \(\theta\) is \(\pi_0(\theta)\). Define
\[
 r(\theta,u)=\pi_0(\theta)L(y\given\theta,u).
\]
Throughout this finite-dimensional full-space construction, the ambient reference measure is \( \dd\theta\,\dd u \). Thus \(r\) is the base density with respect to \( \dd\theta\,\dd u \), with no additional state reference density unless one is explicitly included in \(r\). If a non-flat state reference density is intended, it should be absorbed into \(r(\theta,u)\); the determinant statements below then apply to that enlarged base density.

\begin{assumption}[Well-posed state equation]
\label{ass:state}
For every \(\theta\in\Theta\), the equation \(c(\theta,u)=0\) has a unique solution \(u=G(\theta)\). The map \(c:\Theta\times U\to\R^q\) is continuously differentiable, \(\D_u c(\theta,G(\theta))\) is nonsingular for all \(\theta\), and \(G\) is continuously differentiable.
\end{assumption}
This is the regular single-branch setting supplied by the implicit function theorem and the submanifold calculus used below \cite{Lee2013}. The single-branch assumption is used only to avoid an additional modelling choice over branches. If finitely many nonsingular branches are present, the same local calculation gives a branch sum; this case is recorded in Appendix~\ref{app:multiple_roots}.

Under Assumption~\ref{ass:state}, the reduced posterior is
\begin{equation}
 \pi_{\rm red}(\theta\given y)
 =
 \frac{r(\theta,G(\theta))}{Z_{\rm red}},
 \qquad
 Z_{\rm red}=\int_\Theta r(\theta,G(\theta))\dd\theta .
 \label{eq:red_posterior}
\end{equation}
The corresponding graph-lifted posterior is the push-forward of \(\pi_{\rm red}(\theta\given y)\dd\theta\) under the map
\begin{equation}
 \Phi(\theta)=(\theta,G(\theta)).
\end{equation}
Equivalently, for any bounded test function \(\varphi\),
\begin{equation}
 \int_\Gamma \varphi(\theta,u)\,\pi_\Gamma(\dd\theta,\dd u)
 =
 \int_\Theta \varphi(\theta,G(\theta))\,
 \pi_{\rm red}(\theta\given y)\dd\theta .
 \label{eq:graph_pushforward}
\end{equation}
This is the full-space posterior that exactly represents the reduced Bayesian inverse problem.

The same measure can be written, using standard surface-measure and coarea notation \cite{Federer1969,EvansGariepy2015}, as a density with respect to surface measure on
\begin{equation}
 \Gamma = \{(\theta,u)\in\Theta\times U: c(\theta,u)=0\}.
\end{equation}
Let \(\dd\sigma_\Gamma\) denote the Euclidean surface measure on \(\Gamma\), and let
\begin{equation}
 J_c(\theta,u)=\D_{(\theta,u)}c(\theta,u)=\begin{bmatrix}\D_\theta c(\theta,u)&\D_u c(\theta,u)\end{bmatrix}.
\end{equation}
On \(\Gamma\), differentiating \(c(\theta,G(\theta))=0\) gives
\begin{equation}
 \D_\theta c(\theta,G(\theta)) + \D_u c(\theta,G(\theta))\D G(\theta)=0.
 \label{eq:implicit_derivative}
\end{equation}
The surface element associated with the graph parametrisation is
\begin{equation}
 \dd\sigma_\Gamma
 =
 \sdet{I_p + \D G(\theta)^T\D G(\theta)}\,\dd\theta .
 \label{eq:surface_element}
\end{equation}
Using \eqref{eq:implicit_derivative}, one obtains
\begin{equation}
 \sdet{J_cJ_c^T}
 =
 \abs{\det \D_u c}\,
 \sdet{I_p + \D G^T\D G}
 \qquad\text{on }\Gamma .
 \label{eq:gram_identity}
\end{equation}
Hence the graph-lifted reduced posterior has surface density
\begin{equation}
 \frac{\dd\pi_\Gamma}{\dd\sigma_\Gamma}(\theta,u)
 =
 \frac{1}{Z_{\rm red}}
 r(\theta,u)
 \frac{\abs{\det\D_u c(\theta,u)}}{\sdet{J_c(\theta,u)J_c(\theta,u)^T}},
 \qquad (\theta,u)\in\Gamma .
 \label{eq:graph_surface_density}
\end{equation}

Equation \eqref{eq:graph_surface_density} should be compared with the measure obtained by formally conditioning an ambient density \(r(\theta,u)\dd\theta\dd u\) on the event \(c(\theta,u)=0\), a distinction closely related to conditioning by disintegration \cite{ChangPollard1997}. Using the coarea formula \cite{Federer1969,EvansGariepy2015}, the distribution with density proportional to
\begin{equation}
 r(\theta,u)\delta(c(\theta,u))
\end{equation}
has surface density
\begin{equation}
 \frac{\dd\pi_{\rm res}}{\dd\sigma_\Gamma}(\theta,u)
 \propto
 r(\theta,u)
 \frac{1}{\sdet{J_c(\theta,u)J_c(\theta,u)^T}}.
 \label{eq:residual_surface_density}
\end{equation}
This is the zero-noise residual posterior. It is a legitimate Bayesian target if the residual itself is modelled as an observed noisy quantity. It is not, in general, the graph-lifted reduced posterior. The two surface densities differ by the factor \(\abs{\det\D_u c}\). The factor \(\sdet{J_c(\theta,u)J_c(\theta,u)^T}^{-1}\) is the classical coarea/Fixman surface-density factor for the residual-noise constrained target \cite{Fixman1974,ZappaHolmesCerfonGoodman2018,XuZengPaisleyZhao2026}. The graph-lifted reduced posterior differs from this residual-noise target by the additional state-Jacobian factor \(\abs{\det \D_u c}\).

\begin{remark}[A modeling distinction, not a contradiction]
The graph posterior and the residual posterior answer different questions. The graph posterior represents the usual reduced inverse problem in which \(u\) is a deterministic state determined by \(\theta\). The residual posterior represents a model in which the residual \(c(\theta,u)\) is itself observed with small noise in the ambient \((\theta,u)\)-space. Neither is intrinsically wrong. The point is that they should not be interchanged without accounting for the induced reference measure.
\end{remark}

For the theorem, write
\[
 J_u(\theta,u)=\abs{\det \D_u c(\theta,u)}.
\]
The following assumption is a finite-dimensional sufficient condition for the Laplace/coarea limit used below. It replaces informal phrases such as ``exponentially negligible'' and ``uniformly regular local coordinates'' by conditions that can be checked or deliberately strengthened in examples.

\begin{assumption}[Uniform residual coordinates and dominated tails]
\label{ass:uniform_residual_coordinates}
Let \(r:\Theta\times U\to[0,\infty)\) be measurable. There exist \(\delta>0\) and a measurable family of neighbourhoods \(N_\theta\subset U\), with \(G(\theta)\in N_\theta\), such that the following conditions hold.

\begin{enumerate}[leftmargin=*]
 \item \textbf{Uniform local residual coordinates.}
 For every \(\theta\in\Theta\), the map
 \[
  u\mapsto c(\theta,u)
 \]
 is a \(C^1\)-diffeomorphism from \(N_\theta\) onto \(B_\delta(0)\subset\R^q\).
 Its inverse is denoted by
 \[
  \psi_\theta:B_\delta(0)\to N_\theta,
  \qquad
  \psi_\theta(0)=G(\theta).
 \]
 The tube
 \[
  \mathcal N=\{(\theta,u):u\in N_\theta\}
 \]
 is measurable.

 \item \textbf{Regularity of the inverse coordinates.}
 The map \((\theta,v)\mapsto \psi_\theta(v)\) is measurable on
 \(\Theta\times B_\delta(0)\). For almost every \(\theta\),
 \[
  r(\theta,\psi_\theta(v))\to r(\theta,G(\theta))
  \quad\text{and}\quad
  J_u(\theta,\psi_\theta(v))\to J_u(\theta,G(\theta))
  \quad\text{as }v\to0.
 \]
 Moreover \(J_u(\theta,\psi_\theta(v))>0\) for all \(v\in B_\delta(0)\).

 \item \textbf{Negligible mass outside the tube.}
 As \(\rho\to\infty\),
 \begin{equation}
  \rho^{q/2}
  \int_\Theta\int_{U\setminus N_\theta}
  r(\theta,u)
  \exp\!\left[-\frac{\rho}{2}\norm{c(\theta,u)}^2\right]
  \dd u\,\dd\theta
  \longrightarrow 0,
  \label{eq:assumption_tail_uncorrected}
 \end{equation}
 and
 \begin{equation}
  \rho^{q/2}
  \int_\Theta\int_{U\setminus N_\theta}
  r(\theta,u)J_u(\theta,u)
  \exp\!\left[-\frac{\rho}{2}\norm{c(\theta,u)}^2\right]
  \dd u\,\dd\theta
  \longrightarrow 0 .
  \label{eq:assumption_tail_corrected}
 \end{equation}

 \item \textbf{Dominated convergence after residual scaling.}
 There exist \(\rho_0>0\) and integrable functions
 \(H_{\rm res},H_{\rm graph}\in L^1(\Theta\times\R^q)\) such that, for all
 \(\rho\ge \rho_0\),
 \begin{align}
 &r(\theta,\psi_\theta(z/\sqrt{\rho}))
 J_u(\theta,\psi_\theta(z/\sqrt{\rho}))^{-1}
 \exp\!\left[-\frac{1}{2}\norm{z}^2\right]
 \mathbf 1_{\{\norm{z}<\delta\sqrt{\rho}\}}
 \le H_{\rm res}(\theta,z),
 \label{eq:assumption_dom_uncorrected}\\
 &r(\theta,\psi_\theta(z/\sqrt{\rho}))
 \exp\!\left[-\frac{1}{2}\norm{z}^2\right]
 \mathbf 1_{\{\norm{z}<\delta\sqrt{\rho}\}}
 \le H_{\rm graph}(\theta,z).
 \label{eq:assumption_dom_corrected}
 \end{align}

 \item \textbf{Finite nonzero limiting normalising constants.}
 The limiting constants
 \begin{equation}
  Z_{\rm res}
   =
  \int_\Theta
  r(\theta,G(\theta))
  J_u(\theta,G(\theta))^{-1}
  \dd\theta
  \label{eq:Z_res_explicit}
 \end{equation}
 and
 \begin{equation}
  Z_{\rm red}
  =
  \int_\Theta
  r(\theta,G(\theta))
  \dd\theta
  \label{eq:Z_red_explicit}
 \end{equation}
 satisfy
 \[
  0<Z_{\rm res}<\infty,
  \qquad
  0<Z_{\rm red}<\infty.
 \]
\end{enumerate}
\end{assumption}

\begin{remark}[Compact sufficient conditions]
Assumption~\ref{ass:uniform_residual_coordinates} is stated in dominated-tail form to avoid imposing compactness. A simpler sufficient condition is obtained when \(\Theta\) is compact, \(r\) and \(J_u\) are continuous and bounded on a uniform residual tube, the local diffeomorphisms are uniform in \(\theta\), and the residual is bounded away from zero outside the tube, up to integrable tails in \(u\). The theorem below uses the more explicit dominated-tail formulation because many discretised inverse problems are naturally posed on noncompact parameter or state spaces.
\end{remark}

\section{Penalty limits and the determinant correction}
\label{sec:penalty}

We now derive the limiting parameter posterior induced by a Gaussian residual penalty. For \(\rho>0\), define the naive full-space penalty density
\begin{equation}
 \pi_\rho(\theta,u)
 =
 \frac{1}{Z_\rho}
 r(\theta,u)
 \exp\!\left[-\frac{\rho}{2}\norm{c(\theta,u)}^2\right].
 \label{eq:pi_rho}
\end{equation}
The parameter marginal is
\begin{equation}
 \pi^\theta_\rho(\theta)
 \propto
 \int_U r(\theta,u)
 \exp\!\left[-\frac{\rho}{2}\norm{c(\theta,u)}^2\right]\dd u.
 \label{eq:theta_marginal_penalty}
\end{equation}

\begin{theorem}[Penalty limit]
\label{thm:penalty_limit}
Suppose Assumptions~\ref{ass:state} and \ref{ass:uniform_residual_coordinates} hold, and suppose that the finite-\(\rho\) normalising constants for the densities below are finite and positive for all sufficiently large \(\rho\). Let \(\mu_\rho^\theta\) be the \(\theta\)-marginal of the naive full-space penalty density
\begin{equation}
 \pi_\rho(\theta,u)
 =
 \frac{1}{Z_\rho}
 r(\theta,u)
 \exp\!\left[-\frac{\rho}{2}\norm{c(\theta,u)}^2\right].
 \label{eq:pi_rho_revised}
\end{equation}
Then \(\mu_\rho^\theta\) converges weakly, as \(\rho\to\infty\), to the probability measure on \(\Theta\) with density
\begin{equation}
 \pi^\theta_{\rm res}(\theta)
 =
 \frac{1}{Z_{\rm res}}
 r(\theta,G(\theta))
 J_u(\theta,G(\theta))^{-1}
 =
 \frac{1}{Z_{\rm res}}
 r(\theta,G(\theta))
 \abs{\det\D_u c(\theta,G(\theta))}^{-1}.
 \label{eq:res_limit_density}
\end{equation}

Let \(\widetilde\mu_\rho^\theta\) be the \(\theta\)-marginal of the determinant-corrected finite-penalty density
\begin{equation}
 \widetilde{\pi}_\rho(\theta,u)
 =
 \frac{1}{\widetilde Z_\rho}
 r(\theta,u)
 J_u(\theta,u)
 \exp\!\left[-\frac{\rho}{2}\norm{c(\theta,u)}^2\right].
 \label{eq:corrected_pi_rho}
\end{equation}
Then \(\widetilde\mu_\rho^\theta\) converges weakly to the reduced posterior
\begin{equation}
 \pi_{\rm red}(\theta\given y)
 =
 \frac{1}{Z_{\rm red}}r(\theta,G(\theta)).
 \label{eq:corrected_limit_density_revised}
\end{equation}
All measures and determinants in this theorem are finite-dimensional discretisation-level objects; no infinite-dimensional determinant or mesh-refinement limit is asserted.
\end{theorem}

\begin{proof}
It suffices to test the \(\theta\)-marginals against an arbitrary bounded
continuous function \(h:\Theta\to\R\). Define the unnormalised naive marginal
integral
\begin{equation}
 I_\rho(h)
 =
 \int_\Theta h(\theta)
 \int_U
 r(\theta,u)
 \exp\!\left[-\frac{\rho}{2}\norm{c(\theta,u)}^2\right]
 \dd u\,\dd\theta .
 \label{eq:I_rho_h_revised}
\end{equation}
Split \(I_\rho(h)=I_{\rho,{\rm tube}}(h)+I_{\rho,{\rm out}}(h)\), where the tube contribution integrates over \(N_\theta\) and the outside contribution integrates over \(U\setminus N_\theta\). By
\eqref{eq:assumption_tail_uncorrected},
\[
 \rho^{q/2}\abs{I_{\rho,{\rm out}}(h)}
 \le
 \norm{h}_\infty
 \rho^{q/2}
 \int_\Theta\int_{U\setminus N_\theta}
 r(\theta,u)
 \exp\!\left[-\frac{\rho}{2}\norm{c(\theta,u)}^2\right]
 \dd u\,\dd\theta
 \longrightarrow 0 .
\]

On the tube, use the residual coordinate \(v=c(\theta,u)\). By Assumption~\ref{ass:uniform_residual_coordinates}, \(u=\psi_\theta(v)\) and
\[
 \dd u
 =
 J_u(\theta,\psi_\theta(v))^{-1}\dd v .
\]
Therefore
\begin{align}
 I_{\rho,{\rm tube}}(h)
 &=
 \int_\Theta h(\theta)
 \int_{B_\delta(0)}
 r(\theta,\psi_\theta(v))
 J_u(\theta,\psi_\theta(v))^{-1}
 \exp\!\left[-\frac{\rho}{2}\norm{v}^2\right]
 \dd v\,\dd\theta .
 \label{eq:I_rho_tube_v}
\end{align}
Substituting \(z=\sqrt{\rho}\,v\) gives the exact identity
\begin{align}
 \rho^{q/2}I_{\rho,{\rm tube}}(h)
 &=
 \int_\Theta h(\theta)
 \int_{\norm{z}<\delta\sqrt{\rho}}
 r(\theta,\psi_\theta(z/\sqrt{\rho}))
 J_u(\theta,\psi_\theta(z/\sqrt{\rho}))^{-1}
 \exp\!\left[-\frac{1}{2}\norm{z}^2\right]
 \dd z\,\dd\theta .
 \label{eq:I_rho_scaled_z}
\end{align}
For almost every \((\theta,z)\), the integrand in
\eqref{eq:I_rho_scaled_z} converges to
\[
 h(\theta)
 r(\theta,G(\theta))
 J_u(\theta,G(\theta))^{-1}
 \exp\!\left[-\frac{1}{2}\norm{z}^2\right].
\]
The domination condition \eqref{eq:assumption_dom_uncorrected} and boundedness of \(h\) justify dominated convergence. Hence
\begin{align}
 \rho^{q/2}I_\rho(h)
 &\longrightarrow
 \int_\Theta h(\theta)
 r(\theta,G(\theta))
 J_u(\theta,G(\theta))^{-1}
 \dd\theta
 \int_{\R^q}
 \exp\!\left[-\frac{1}{2}\norm{z}^2\right]\dd z \notag\\
 &=
 (2\pi)^{q/2}
 \int_\Theta h(\theta)
 r(\theta,G(\theta))
 J_u(\theta,G(\theta))^{-1}
 \dd\theta .
 \label{eq:I_rho_limit_revised}
\end{align}
Taking \(h\equiv1\) gives
\[
 \rho^{q/2}I_\rho(1)
 \longrightarrow
 (2\pi)^{q/2}Z_{\rm res},
\]
with \(0<Z_{\rm res}<\infty\) by Assumption~\ref{ass:uniform_residual_coordinates}. Therefore
\[
 \int_\Theta h(\theta)\,\mu_\rho^\theta(\dd\theta)
 =
 \frac{I_\rho(h)}{I_\rho(1)}
 \longrightarrow
 \frac{1}{Z_{\rm res}}
 \int_\Theta h(\theta)
 r(\theta,G(\theta))
 J_u(\theta,G(\theta))^{-1}
 \dd\theta ,
\]
which proves weak convergence of the naive penalty marginal to
\eqref{eq:res_limit_density}.

For the corrected density, define
\begin{equation}
 \widetilde I_\rho(h)
 =
 \int_\Theta h(\theta)
 \int_U
 r(\theta,u)J_u(\theta,u)
 \exp\!\left[-\frac{\rho}{2}\norm{c(\theta,u)}^2\right]
 \dd u\,\dd\theta .
 \label{eq:I_tilde_rho_h}
\end{equation}
Again split
\(\widetilde I_\rho(h)=\widetilde I_{\rho,{\rm tube}}(h)+
\widetilde I_{\rho,{\rm out}}(h)\). The outside contribution satisfies
\[
 \rho^{q/2}\abs{\widetilde I_{\rho,{\rm out}}(h)}
 \le
 \norm{h}_\infty
 \rho^{q/2}
 \int_\Theta\int_{U\setminus N_\theta}
 r(\theta,u)J_u(\theta,u)
 \exp\!\left[-\frac{\rho}{2}\norm{c(\theta,u)}^2\right]
 \dd u\,\dd\theta
 \longrightarrow 0
\]
by \eqref{eq:assumption_tail_corrected}.

On the tube, the same change of variables gives
\begin{align}
 \widetilde I_{\rho,{\rm tube}}(h)
 &=
 \int_\Theta h(\theta)
 \int_{B_\delta(0)}
 r(\theta,\psi_\theta(v))
 J_u(\theta,\psi_\theta(v))
 J_u(\theta,\psi_\theta(v))^{-1}
 \exp\!\left[-\frac{\rho}{2}\norm{v}^2\right]
 \dd v\,\dd\theta \notag\\
 &=
 \int_\Theta h(\theta)
 \int_{B_\delta(0)}
 r(\theta,\psi_\theta(v))
 \exp\!\left[-\frac{\rho}{2}\norm{v}^2\right]
 \dd v\,\dd\theta .
 \label{eq:I_tilde_cancel}
\end{align}
Thus the state-Jacobian determinant in the corrected density cancels the Jacobian of the residual-coordinate transformation. With \(z=\sqrt{\rho}\,v\),
\begin{align}
 \rho^{q/2}\widetilde I_{\rho,{\rm tube}}(h)
 &=
 \int_\Theta h(\theta)
 \int_{\norm{z}<\delta\sqrt{\rho}}
 r(\theta,\psi_\theta(z/\sqrt{\rho}))
 \exp\!\left[-\frac{1}{2}\norm{z}^2\right]
 \dd z\,\dd\theta .
 \label{eq:I_tilde_scaled_z}
\end{align}
The integrand converges pointwise almost everywhere to
\[
 h(\theta)r(\theta,G(\theta))
 \exp\!\left[-\frac{1}{2}\norm{z}^2\right],
\]
and \eqref{eq:assumption_dom_corrected} gives the required domination.
Consequently
\begin{align}
 \rho^{q/2}\widetilde I_\rho(h)
 &\longrightarrow
 (2\pi)^{q/2}
 \int_\Theta h(\theta)r(\theta,G(\theta))\dd\theta .
 \label{eq:I_tilde_limit_revised}
\end{align}
Taking \(h\equiv1\) gives
\[
 \rho^{q/2}\widetilde I_\rho(1)
 \longrightarrow
 (2\pi)^{q/2}Z_{\rm red},
\]
with \(0<Z_{\rm red}<\infty\). Therefore
\[
 \int_\Theta h(\theta)\,\widetilde\mu_\rho^\theta(\dd\theta)
 =
 \frac{\widetilde I_\rho(h)}{\widetilde I_\rho(1)}
 \longrightarrow
 \frac{1}{Z_{\rm red}}
 \int_\Theta h(\theta)r(\theta,G(\theta))\dd\theta ,
\]
which proves weak convergence to the reduced posterior
\eqref{eq:corrected_limit_density_revised}.
\end{proof}

\begin{corollary}[Equivalent residuals can give different naive limits]
\label{cor:scaling}
Let \(a:\Theta\to(0,\infty)\) be smooth and define the scalar-rescaled residual
\begin{equation}
 c_a(\theta,u)=a(\theta)c(\theta,u).
\end{equation}
The feasible set \(\{c_a=0\}\) equals \(\{c=0\}\). However, the naive penalty limit associated with \(c_a\) has \(\theta\)-density proportional to
\begin{equation}
 r(\theta,G(\theta))
 a(\theta)^{-q}
 \abs{\det\D_u c(\theta,G(\theta))}^{-1}.
\end{equation}
The determinant-corrected hard-constraint limit is invariant under this rescaling. Indeed,
\[J_{u,a}(\theta,u) = \abs{\det \D_u c_a(\theta,u)} =a(\theta)^q J_u(\theta,u),\]
so applying Theorem~\ref{thm:penalty_limit} to \(c_a\) and the corrected density gives limiting \(\theta\)-density proportional to \(r(\theta,G(\theta))\). At finite \(\rho\), however, the unweighted corrected density is generally not identical to \eqref{eq:corrected_pi_rho}, since it is proportional to
\[ r(\theta,u)\, a(\theta)^qJ_u(\theta,u) \exp\!\left[-\frac{\rho}{2}a(\theta)^2\norm{c(\theta,u)}^2\right].\]
Exact finite-\(\rho\) invariance under algebraic residual transformations is obtained by transforming the residual weight together with the residual, as in Remark~\ref{rem:weighted_rescaling_invariance}.
\end{corollary}

\begin{corollary}[Weighted residual penalties]
\label{cor:weighted_residuals}
Let \(R:\Theta\to\R^{q\times q}\) be smooth, symmetric positive definite, with \(\det R(\theta)>0\) for all \(\theta\). Consider the weighted residual penalty
\begin{equation}
 \exp\!\left[
  -\frac{\rho}{2}c(\theta,u)^T R(\theta)c(\theta,u)
 \right].
 \label{eq:weighted_residual_penalty}
\end{equation}
Equivalently, define the whitened residual
\begin{equation}
 w(\theta,u)=R(\theta)^{1/2}c(\theta,u).
 \label{eq:whitened_residual}
\end{equation}
Assume that the hypotheses of Theorem~\ref{thm:penalty_limit} hold with
\(w\) in place of \(c\). On the graph \(c(\theta,G(\theta))=0\),
\begin{equation}
 \D_u w(\theta,G(\theta))
 =
 R(\theta)^{1/2}\D_u c(\theta,G(\theta)),
 \label{eq:Du_weighted_residual}
\end{equation}
and hence
\begin{equation}
 \abs{\det \D_u w(\theta,G(\theta))}
 =
 \det(R(\theta))^{1/2}
 \abs{\det \D_u c(\theta,G(\theta))}.
 \label{eq:weighted_residual_det}
\end{equation}
Therefore the naive weighted residual penalty converges to a \(\theta\)-density
proportional to
\begin{equation}
 r(\theta,G(\theta))
 \abs{\det \D_u c(\theta,G(\theta))}^{-1}
 \det(R(\theta))^{-1/2}.
 \label{eq:weighted_naive_limit}
\end{equation}
The graph-corrected weighted finite-penalty target has log-density addend
\begin{equation}
 \log\abs{\det \D_u c(\theta,u)}
 +\frac{1}{2}\log\det R(\theta)
 -\frac{\rho}{2}c(\theta,u)^T R(\theta)c(\theta,u),
 \label{eq:weighted_corrected_addend}
\end{equation}
in addition to the user's prior and likelihood log density.
Singular or indefinite weights are outside this corollary.
\end{corollary}

\begin{remark}[Unnormalised penalties versus normalised residual likelihoods]
\label{rem:weighted_normalisation}
Corollary~\ref{cor:weighted_residuals} treats
\[ \exp\!\left[-\frac{\rho}{2}c(\theta,u)^TR(\theta)c(\theta,u)\right]\]
as an unnormalised penalty kernel. If the residual term is instead specified as a normalised Gaussian likelihood with precision \(\rho R(\theta)\), then the likelihood normalisation already contains the factor \(\det R(\theta)^{1/2}\), up to constants independent of \(\theta\). In that case the \(\frac12\log\det R(\theta)\) term should not be added a second time.
\end{remark}

\begin{proof}
Since \(R(\theta)\) does not depend on \(u\), differentiating \(w(\theta,u)=R(\theta)^{1/2}c(\theta,u)\) with respect to \(u\) gives \eqref{eq:Du_weighted_residual}. Taking absolute determinants gives \eqref{eq:weighted_residual_det}. Applying Theorem~\ref{thm:penalty_limit} to the residual coordinate \(w\) gives the naive limit
\[
 r(\theta,G(\theta))
 \abs{\det \D_u w(\theta,G(\theta))}^{-1},
\]
which is exactly \eqref{eq:weighted_naive_limit}. The corrected target for the whitened residual multiplies the weighted penalty by \(\abs{\det \D_u w}\). Using \eqref{eq:weighted_residual_det} yields the log addend \eqref{eq:weighted_corrected_addend}.
\end{proof}

\begin{remark}[Residual-rescaling invariance]
\label{rem:weighted_rescaling_invariance}
Let \(A(\theta)\) be a smooth invertible \(q\times q\) matrix and define
\[
 c_A(\theta,u)=A(\theta)c(\theta,u),
 \qquad
 R_A(\theta)=A(\theta)^{-T}A(\theta)^{-1}.
\]
Then
\begin{equation}
 c_A(\theta,u)^T R_A(\theta)c_A(\theta,u)
 = c(\theta,u)^T c(\theta,u),
 \label{eq:rescaling_penalty_invariance}
\end{equation}
and
\[
 \D_u c_A(\theta,u)=A(\theta)\D_u c(\theta,u).
\]
Moreover
\[
 \det R_A(\theta)=\abs{\det A(\theta)}^{-2},
\]
so the combined graph correction satisfies
\begin{align}
 \log\abs{\det \D_u c_A(\theta,u)}+\frac{1}{2}\log\det R_A(\theta)
 &=
 \log\abs{\det A(\theta)}+\log\abs{\det \D_u c(\theta,u)}-\log\abs{\det A(\theta)} \notag\\
 &=
 \log\abs{\det \D_u c(\theta,u)}.
 \label{eq:combined_rescaling_correction}
\end{align}
Thus the graph-corrected weighted target is invariant under this algebraic residual transformation. This is the mathematical statement implemented by the weighted-residual helper described in Appendix~\ref{app:detcorr_api}.
\end{remark}

\begin{remark}[State-equation determinants]
For a linear state equation \(c(\theta,u)=A(\theta)u-f(\theta)\), the determinant in Theorem~\ref{thm:penalty_limit} is \(\abs{\det A(\theta)}\). Thus the naive penalty limit includes \(\abs{\det A(\theta)}^{-1}\) in the parameter posterior. In large discretised PDEs this term may be expensive and mesh-dependent, but its possible computational cost does not make it mathematically absent. It either belongs to the model, or it must be cancelled by choosing a graph-corrected target or by sampling directly in the reduced parameter space.
\end{remark}

\section{Sampling templates after target specification}
\label{sec:algorithms}

The previous sections specify the finite-dimensional target measures. We now describe how augmented-Lagrangian and splitting ideas can be used without turning residual reduction into an implicit posterior definition. The optimisation primitives are classical in multiplier, augmented-Lagrangian, Douglas--Rachford/ADMM, proximal, and primal-dual methods \cite{Hestenes1969,Rockafellar1974,EcksteinBertsekas1992,BoydParikhChuPeleatoEckstein2011,ParikhBoyd2014,ChambollePock2011}. Here they are used only as proposal geometry, preconditioning, tempering, initialisation, or variational transport. Posterior sampling validity is supplied by an invariant Markov transition, a valid SMC correction, a Metropolis correction on a manifold, or another mechanism that preserves the specified target \cite{Neal2001,Chopin2002,DelMoralDoucetJasra2006, BeskosJasraMuzafferStuart2015}. The templates below are therefore not claims that augmented-Lagrangian residual updates alone sample the posterior.

\subsection{A corrected augmented-Lagrangian path}

For fixed penalty \(\rho>0\) and multiplier \(\lambda\in\R^q\), define an
unnormalised determinant-corrected augmented-Lagrangian density
\begin{equation}
 \gamma_{\rho,\lambda}(\theta,u)
 =
 r(\theta,u)
 \abs{\det\D_u c(\theta,u)}
 \exp\!\left[-\lambda^T c(\theta,u)-\frac{\rho}{2}\norm{c(\theta,u)}^2\right].
 \label{eq:gamma_rho_lambda}
\end{equation}
For fixed \((\rho,\lambda)\), this is an ordinary density on \(\Theta\times U\). If \(\lambda\) is fixed as \(\rho\to\infty\), the term \(-\lambda^Tc\) shifts the centre of the finite-\(\rho\) residual tube but does not change the leading graph-limit marginal in \(\theta\). In residual coordinates \(v=c(\theta,u)\),
\[ -\lambda^Tv-\frac{\rho}{2}\norm{v}^2
 =
 -\frac{\rho}{2}\norm{v+\lambda/\rho}^2
 +\frac{\norm{\lambda}^2}{2\rho},\]
so the shift is \(O(\rho^{-1})\) and the remaining Gaussian factor is
independent of \(\theta\). The same conclusion holds for multiplier schedules
\(\lambda=\lambda_\rho\) with \(\norm{\lambda_\rho}/\rho\to0\), provided the
completed-square factor remains independent of \(\theta\). If multipliers grow
too rapidly with \(\rho\), or if they are made \(\theta\)-dependent, the
graph-limit marginal may change and must be analysed as part of the declared
target. Thus \eqref{eq:gamma_rho_lambda} can be used as an intermediate density
in a tempering or SMC path, or as the invariant density of a Markov kernel at a
given stage. If \((\rho,\lambda)\) is adapted during a computation, the
adaptation must be accounted for by the surrounding SMC or MCMC scheme.

A generic invariant implementation is as follows.

\begin{center}
\fbox{\begin{minipage}{0.94\linewidth}
\textbf{Algorithm 1: Template for determinant-corrected SMC/MCMC with augmented-Lagrangian intermediate densities or proposals}

\smallskip
\textbf{Input:} prior \(\pi_0\), likelihood \(L\), residual \(c\), schedule
\((\rho_k,\lambda_k)_{k=0}^K\), number of particles \(N\), and Markov kernels
\(K_k\) invariant with respect to \(\gamma_{\rho_k,\lambda_k}\).

\begin{enumerate}[leftmargin=*]
 \item Initialise particles \(x_0^{(i)}=(\theta_0^{(i)},u_0^{(i)})\) from an
 initial density \(\eta_0\). Set weights
 \(w_0^{(i)}\propto \gamma_{\rho_0,\lambda_0}(x_0^{(i)})/
 \eta_0(x_0^{(i)})\).
 \item For \(k=0,\ldots,K-1\):
 \begin{enumerate}[leftmargin=*]
  \item Reweight by
  \[
   w_{k+1}^{(i)}\leftarrow
   w_k^{(i)}
   \frac{\gamma_{\rho_{k+1},\lambda_{k+1}}(x_k^{(i)})}
      {\gamma_{\rho_k,\lambda_k}(x_k^{(i)})}.
  \]
  \item Resample if the effective sample size is below a chosen threshold.
  \item Move each particle with one or more steps of an MCMC kernel
  \(K_{k+1}\) that leaves \(\gamma_{\rho_{k+1},\lambda_{k+1}}\) invariant.
  Augmented-Lagrangian, ADMM, Gauss--Newton, stochastic-gradient, or particle
  transport steps may be used inside this move only as proposals unless they
  themselves preserve the target.
 \end{enumerate}
 \item Return the weighted or resampled particles at \(k=K\). Large
 \(\rho_K\) approximates the graph-corrected hard-constraint target; finite
 \(\rho_K\) represents a softened model with the explicitly stated density
 \(\gamma_{\rho_K,\lambda_K}\).
\end{enumerate}
\end{minipage}}
\end{center}

The proposal inside \(K_k\) can be built from an ADMM step, a Gauss--Newton step, a stochastic gradient step, a transport map, or an SVGD particle update. The invariant measure is supplied by the Metropolis--Hastings correction \cite{MetropolisRosenbluthTellerTeller1953,Hastings1970}, by an HMC-type transition \cite{GirolamiCalderhead2011,DuaneKennedyPendletonRoweth1987,Neal2011}, or by another transition that preserves \(\gamma_{\rho_k,\lambda_k}\). For example, a proposal \(q_k(x,\dd x')\) is accepted with probability
\begin{equation}
 \alpha_k(x,x')
 =
 1\wedge
 \frac{\gamma_{\rho_k,\lambda_k}(x')q_k(x',x)}
    {\gamma_{\rho_k,\lambda_k}(x)q_k(x,x')}.
 \label{eq:mh_acceptance}
\end{equation}
This distinction is essential: an augmented-Lagrangian or ADMM step may make an excellent proposal, preconditioner, mutation, or warm start, but posterior correctness is a property of the invariant transition and its declared target density.

\subsection{Direct manifold sampling}

An alternative, related to constrained geometric integration and manifold Monte
Carlo
\cite{ByrneGirolami2013,ZappaHolmesCerfonGoodman2018,GrahamThieryBeskos2022,RyckaertCiccottiBerendsen1977,Andersen1983,HairerLubichWanner2006},
is to sample directly on the constraint manifold
\begin{equation}
 \Gamma=\{(\theta,u):c(\theta,u)=0\}.
\end{equation}
The graph-lifted reduced posterior has surface density
\eqref{eq:graph_surface_density}. A constrained HMC or RATTLE-type method
introduces momentum \(p\) in the tangent space
\begin{equation}
 T_{(\theta,u)}\Gamma
 =
 \{v\in\R^{p+q}:J_c(\theta,u)v=0\},
\end{equation}
simulates Hamiltonian dynamics constrained to \(\Gamma\), projects numerical errors back to the manifold, and applies a Metropolis correction for the target surface density. Augmented-Lagrangian or ADMM curvature can still be used to choose the mass matrix or local preconditioner. The projection and acceptance steps are what preserve the desired measure; reducing the residual alone is not a substitute for this correction.

\begin{center}
\fbox{\begin{minipage}{0.94\linewidth}
\textbf{Algorithm 2: Template for graph-target manifold MCMC with augmented-Lagrangian preconditioning}

\smallskip
\textbf{Input:} feasible initial point \(x_0=(\theta_0,u_0)\in\Gamma\), target
surface density
\[
 \bar\pi_\Gamma(x)
 \propto
 r(x)
 \frac{\abs{\det\D_u c(x)}}{\sdet{J_c(x)J_c(x)^T}},
 \qquad x\in\Gamma,
\]
constraint Jacobian \(J_c\), and a positive definite mass matrix \(M(x)\)
possibly derived from augmented-Lagrangian normal equations.

\begin{enumerate}[leftmargin=*]
 \item Sample momentum \(p\in T_x\Gamma\) from the tangent Gaussian induced by
 \(M(x)\).
 \item Integrate constrained Hamiltonian dynamics using a reversible projection
 scheme.
 \item Project the endpoint to \(\Gamma\) and the momentum to \(T_x\Gamma\).
 \item Accept or reject using the Metropolis rule for \(\bar\pi_\Gamma\) and
 the constrained proposal density.
\end{enumerate}
\end{minipage}}
\end{center}

This approach avoids taking \(\rho\to\infty\), but it requires feasible initialisation, a declared surface density, and reliable projection onto \(\Gamma\). In large PDE inverse problems, these requirements may be demanding; this is precisely where full-space augmented-Lagrangian proposals or tempering paths can be useful.

\subsection{Where particle variational methods fit}

Particle variational methods, including SVGD and recent dual augmented-Lagrangian variants for constrained inverse problems \cite{LiuWang2016,SiahkoohiAghazadeGholami2026,AghazadeSiahkoohiGholami2026}, can be useful in this framework. They can approximate a finite-penalty density, construct a transport proposal, initialise an SMC population, or identify a low-dimensional likelihood-informed subspace. The distinction is terminological and mathematical. A deterministic KL-descent particle flow is a variational approximation to a target density. It becomes an exact or asymptotically exact posterior sampling component only when embedded in a procedure with the appropriate invariant measure or sequential correction, such as a Metropolis correction, a valid SMC sampler, or another exact transition. The determinant correction is orthogonal to the choice of particle or Markov method: it specifies which finite-dimensional target density the method should approximate or preserve.

\section{Analytic example: same feasible set, different posterior}
\label{sec:analytic_example}

The determinant effect is already visible in a scalar nonlinear inverse problem. Let
\begin{equation}
 \theta\sim N(0,1),
 \qquad
 y\given u\sim N(u,\sigma_y^2),
 \qquad
 u=\theta^2.
\end{equation}
The reduced posterior is
\begin{equation}
 \pi_{\rm red}(\theta\given y)
 \propto
 \exp\!\left[-\frac{\theta^2}{2}\right]
 \exp\!\left[-\frac{(y-\theta^2)^2}{2\sigma_y^2}\right].
 \label{eq:scalar_reduced}
\end{equation}
Now represent the same feasible set by the residual
\begin{equation}
 c_a(\theta,u)=a(\theta)(u-\theta^2),
 \qquad
 a(\theta)=\exp(\theta).
 \label{eq:scalar_residual}
\end{equation}
The condition \(c_a=0\) is exactly equivalent to \(u=\theta^2\). However, Theorem~\ref{thm:penalty_limit} gives the naive penalty limit
\begin{equation}
 \pi_{\rm naive}(\theta\given y)
 \propto
 \exp\!\left[-\theta\right]
 \exp\!\left[-\frac{\theta^2}{2}\right]
 \exp\!\left[-\frac{(y-\theta^2)^2}{2\sigma_y^2}\right].
 \label{eq:scalar_naive}
\end{equation}
Thus residual scaling has changed the posterior, despite leaving the physical constraint unchanged. The determinant-corrected penalty multiplies by \(\abs{\D_u c_a}=\exp(\theta)\) and recovers \eqref{eq:scalar_reduced}.

Figure~\ref{fig:analytic_counterexample} expands the scalar example into a four-part diagnostic. Panels~(a)--(c) visualise the mechanism in Theorem~\ref{thm:penalty_limit}: the feasible graph is unchanged, but the volume of a small residual tube changes under residual rescaling, and the state-Jacobian determinant cancels that volume change when the intended target is the graph-lifted reduced posterior. Panel~(d) is the software counterpart of Corollary~\ref{cor:weighted_residuals} and Remark~\ref{rem:weighted_rescaling_invariance}. It compares \(c\) with an algebraically equivalent residual \(A(\theta)c\) using the compensating weight \(R_A(\theta)=A(\theta)^{-T}A(\theta)^{-1}\). The corrected targets agree to roundoff because both the state-Jacobian determinant and the weighted-residual volume term are included. The negative control fails because the correction is deliberately omitted.

\begin{figure}[H]
 \centering
 \includegraphics[width=0.95\linewidth]{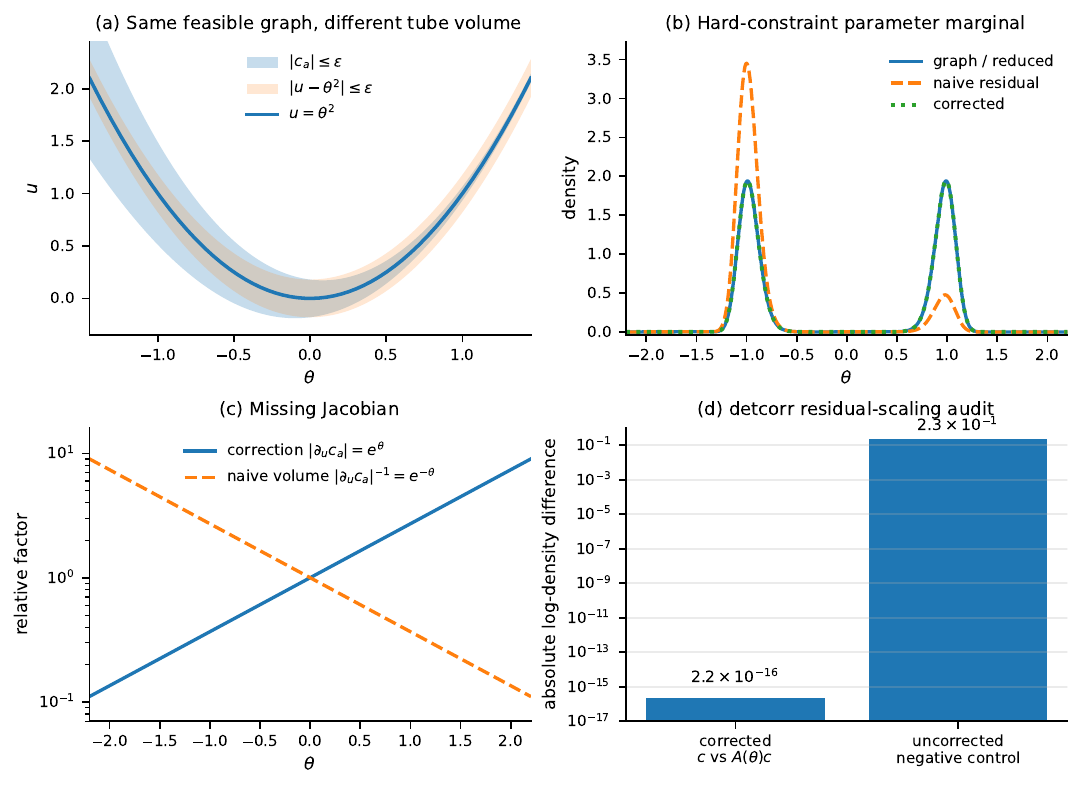}
 \caption{Analytic scalar example and software-level residual-scaling diagnostic. (a) Scaling \(c_a=\exp(\theta)(u-\theta^2)\) leaves the graph unchanged but changes residual-tube volume. (b) The naive hard-constraint limit is reweighted by \(\exp(-\theta)\), while the corrected limit matches the graph posterior. (c) The missing residual-volume factor and compensating determinant correction. (d) The \detcorr{} residual-scaling audit: corrected equivalent targets agree to roundoff because the state-Jacobian and weighted-residual volume corrections are both included, while the uncorrected negative control differs at \(O(10^{-1})\).}
 \label{fig:analytic_counterexample}
\end{figure}

\section{PDE-constrained benchmark protocol}
\label{sec:pde_protocol}

A natural numerical benchmark is a one-dimensional elliptic coefficient inverse problem, a standard setting for PDE-constrained Bayesian inversion and uncertainty quantification \cite{BuiThanhGirolami2014,BuiThanhNguyen2016, FlathWilcoxAkcelikHillVanBloemenWaandersGhattas2011, MartinWilcoxBursteddeGhattas2012,PetraMartinStadlerGhattas2014, CuiMartinMarzoukSolonenSpantini2014,CuiLawMarzouk2016}. Let \(m(x)\) be an unknown log-conductivity field and let \(u(x)\) solve
\begin{equation}
 -\frac{\dd}{\dd x}\left(\exp(m(x))\frac{\dd u}{\dd x}\right)=f(x),
 \qquad x\in(0,1),
 \qquad u(0)=u(1)=0.
 \label{eq:elliptic_pde}
\end{equation}
Noisy observations are
\begin{equation}
 y = H u + \eta,
 \qquad
 \eta\sim N(0,\Gamma_{\rm obs}).
\end{equation}
After finite-element or finite-difference discretisation, the state equation has the form
\begin{equation}
 c(m,u)=A(m)u-f=0.
 \label{eq:discrete_elliptic_constraint}
\end{equation}
For the validation-scale benchmark reported below, we restrict the coefficient to a one-dimensional parametric family
\begin{equation}
 m_\theta(x)=m_{\rm ref}(x)+\theta\phi(x),
 \qquad \theta\in\R,
 \label{eq:scalar_coefficient_family}
\end{equation}
and write \(A(\theta)=A(m_\theta)\). The \(1001\)-point grid in Table~\ref{tab:elliptic_benchmark_summary} is therefore a deterministic grid over the scalar coefficient parameter \(\theta\), not a grid over the full
field \(m(x)\).
The reduced posterior is
\begin{equation}
 \pi_{\rm red}(m\given y)
 \propto
 \pi_0(m)
 \exp\!\left[-\frac{1}{2}\norm{H A(m)^{-1}f-y}_{\Gamma_{\rm obs}^{-1}}^2\right].
 \label{eq:elliptic_reduced_posterior}
\end{equation}
The naive full-space penalty posterior is
\begin{equation}
 \pi_\rho(m,u\given y)
 \propto
 \pi_0(m)
 \exp\!\left[-\frac{1}{2}\norm{Hu-y}_{\Gamma_{\rm obs}^{-1}}^2\right]
 \exp\!\left[-\frac{\rho}{2}\norm{A(m)u-f}^2\right].
 \label{eq:elliptic_naive_penalty}
\end{equation}
Theorem~\ref{thm:penalty_limit} predicts that its limiting \(m\)-marginal is proportional to
\begin{equation}
 \pi_{\rm red}(m\given y)\,\abs{\det A(m)}^{-1}.
 \label{eq:elliptic_naive_limit}
\end{equation}
The graph-corrected penalty is obtained by multiplying \eqref{eq:elliptic_naive_penalty} by \(\abs{\det A(m)}\).

A full sampler benchmark, beyond the deterministic validation check below, should compare:
\begin{enumerate}[leftmargin=*]
 \item a reduced-space reference sampler for \eqref{eq:elliptic_reduced_posterior};
 \item the naive full-space penalty sampler \eqref{eq:elliptic_naive_penalty};
 \item the determinant-corrected full-space sampler;
 \item an augmented-Lagrangian proposal or particle transport used as a warm start or mutation proposal inside an invariant SMC/MCMC scheme.
\end{enumerate}
The primary diagnostics should be posterior marginal agreement with the reduced reference, posterior predictive calibration, effective sample size, split-\(\hat R\), posterior-quantile or simulation-based calibration checks, and the residual norm \(\norm{A(m)u-f}\) \cite{GelmanRubin1992,VehtariGelmanSimpsonCarpenterBurkner2021, CookGelmanRubin2006, ModrakMoonKimBurknerHuurreFaltejskovaGelmanVehtari2025}. Residual convergence alone is not a sufficient diagnostic, because both naive and corrected penalties concentrate on the same feasible set.

\subsection{Code-generated benchmark diagnostics}
\label{subsec:code_generated_benchmark_diagnostics}

The benchmark is implemented in \detcorr{} as a deterministic, validation-scale scalar-parametric problem using \eqref{eq:scalar_coefficient_family}. The reduced route evaluates the induced posterior \(\pi_{\rm red}(\theta\given y)\) on a one-dimensional grid in \(\theta\). The full-space routes analytically integrate over the linear-Gaussian state variable at finite penalty strength \(\rho=10^5\), while sampler-facing pointwise checks use the exact determinant route inside the corrected target helper. The example is deliberately small enough for continuous integration; it is a calibration check for the measure correction in Theorem~\ref{thm:penalty_limit}, not a claim about large-scale PDE solver performance.

Figure~\ref{fig:elliptic_benchmark_routes} and Table~\ref{tab:elliptic_benchmark_summary} connect the finite-dimensional elliptic benchmark directly to Theorem~\ref{thm:penalty_limit}. For the affine state equation \(c(m,u)=A(m)u-f\), the missing volume term in the naive full-space route is \(\abs{\det A(m)}^{-1}\). The determinant-corrected full-space marginal agrees with the reduced posterior to finite-penalty and grid-quadrature tolerance, whereas the naive full-space marginal is shifted by the missing \(\log\abs{\det A(\theta)}\) term. The pointwise \detcorr{} audit evaluated four off-graph states and found maximum absolute log-determinant error below displayed precision. These diagnostics compare posterior marginals and calibration quantities; residual norms are reported separately because feasibility is necessary for a hard-constraint limit but not sufficient for posterior correctness.

\begin{figure}[H]
 \centering
 \includegraphics[width=0.72\linewidth]{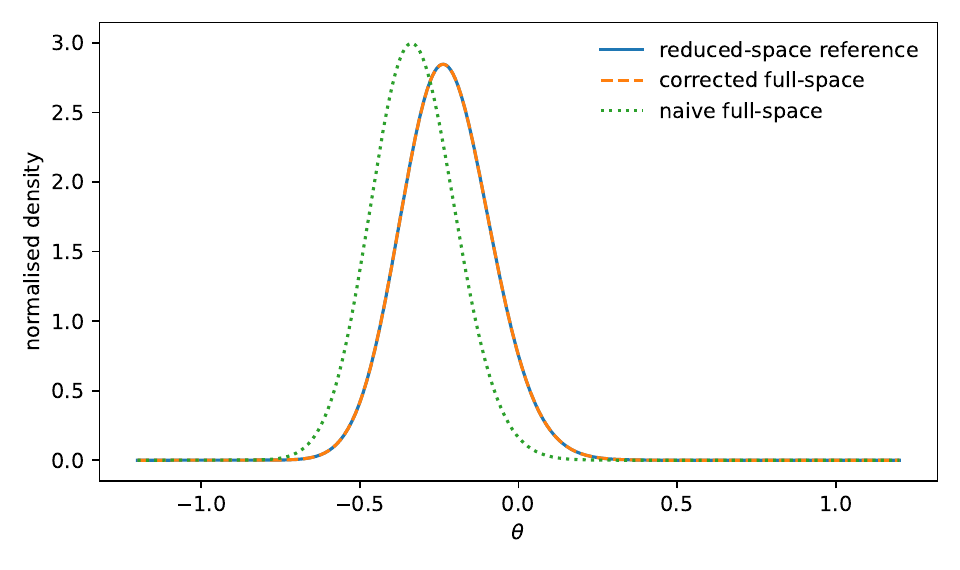}
 \caption{Section~\ref{sec:pde_protocol} validation-scale benchmark generated from the \detcorr{} implementation. The corrected full-space finite-penalty marginal is visually indistinguishable from the reduced-space reference, as predicted by Theorem~\ref{thm:penalty_limit}, while the naive full-space marginal is biased by the missing state-Jacobian volume factor.}
 \label{fig:elliptic_benchmark_routes}
\end{figure}

\begin{table}[H]
\centering
\caption{Section~\ref{sec:pde_protocol} benchmark summaries. All values are computed on a \(1001\)-point grid in the scalar coefficient parameter \(\theta\), with \(q=6\) state variables and \(\rho=10^5\).}
\label{tab:elliptic_benchmark_summary}
\begin{tabular}{lrrrrr}
\toprule
Route & Mean & Variance & 5\% quantile & Median & 95\% quantile\\
\midrule
Reduced space & -0.227033 & 0.019986 & -0.454522 & -0.231716 & 0.009897\\
Corrected full space & -0.227037 & 0.019990 & -0.454556 & -0.231717 & 0.009916\\
Naive full space & -0.329464 & 0.017949 & -0.546571 & -0.333164 & -0.106251\\
\bottomrule
\end{tabular}
\end{table}

\begin{table}[H]
\centering
\caption{Numerical diagnostics generated by \detcorr{} for the algebraic residual-scaling test. The two corrected formulations use the same inverse problem with residuals \(c(\theta,u)\) and \(A(\theta)c(\theta,u)\), including both the state-Jacobian determinant and the compensating weighted-residual volume term from Corollary~\ref{cor:weighted_residuals}.}
\label{tab:residual_scaling_diagnostics}
\begin{tabular}{lr}
\toprule
Diagnostic & Value\\
\midrule
Maximum audit-grid log-density difference & \(2.22\times10^{-16}\)\\
Mean negative-control difference without correction & \(2.28\times10^{-1}\)\\
Maximum retained \(\theta\)-draw difference & \(0\)\\
Maximum retained log-density difference & \(1.78\times10^{-15}\)\\
Acceptance rate, both corrected runs & \(0.32325\)\\
Maximum posterior-summary difference & \(0\)\\
Feasibility/calibration logs separately tagged & true\\
\bottomrule
\end{tabular}
\end{table}

The residual-scaling diagnostic is the practical counterpart of Corollary~\ref{cor:scaling}, Corollary~\ref{cor:weighted_residuals}, and Remark~\ref{rem:weighted_rescaling_invariance}. The benchmark evaluates both residual coordinates on an audit grid and then runs the same seeded validation-scale Metropolis driver. With the graph correction and the weighted-residual factor included, the corrected log densities agree to roundoff and the sampled parameter summaries are identical. The negative control, which omits the correction, deliberately fails the equivalence check. This is the intended software-level guard against confusing small residuals with posterior calibration.

\section{Discussion}
\label{sec:discussion}

The theorem has several implications for Bayesian inverse problems.

First, it separates physical feasibility from posterior correctness. A sequence of samples or particles may satisfy \(\norm{c(\theta,u)}\approx0\) while still approximating the wrong constrained measure. This is not a pathology of a particular algorithm; it is a consequence of changing variables from state coordinates to residual coordinates.

Second, it explains why residual scaling matters. In deterministic optimisation, replacing \(c\) by \(A(\theta)c\) often changes conditioning but not the set of feasible minimisers. In Bayesian computation, the same replacement changes the volume of the residual tube around \(\Gamma\). Unless the target is corrected, the limiting posterior depends on this arbitrary choice.

Third, it gives a constructive role to augmented-Lagrangian and splitting methods. These methods can improve proposals, define tempering paths, and produce excellent variational ensembles. A rigorous sampling algorithm should additionally specify the invariant density at each step and include the determinant correction if the intended target is the graph-lifted reduced posterior.

Fourth, it clarifies the finite-dimensional nature of the statement. In a mesh-refined PDE inverse problem, determinants such as \(\det A(m)\) may diverge, vanish, or require renormalisation. Function-space Bayesian inverse problems are usually formulated with respect to carefully chosen Gaussian or Besov reference measures rather than Lebesgue measure \cite{Stuart2010,CotterRobertsStuartWhite2013, BeskosPinskiSanzSernaStuart2011,HairerStuartVollmer2014,Law2014, BuiThanhNguyen2016}. Therefore, Theorem~\ref{thm:penalty_limit} should not be read as a complete infinite-dimensional result. Instead, it is a discretisation-level warning: the posterior induced by a residual penalty depends on the residual reference measure, and this dependence should be controlled before interpreting full-space particles as samples from a reduced posterior.

Finally, there are situations in which the residual posterior \eqref{eq:residual_surface_density} is the desired model. For example, if \(c(\theta,u)\) represents a physical discrepancy or model-error observation with a specified noise law, as in calibration, ABC/model-error interpretations, or residual-noise physics-constrained generative modelling, then the determinant factor in \eqref{eq:res_limit_density} is part of the intended Bayesian model \cite{KennedyOHagan2001,Wilkinson2013,XuZengPaisleyZhao2026}. The correction \eqref{eq:corrected_pi_rho} is needed only when the goal is to reproduce the reduced posterior \eqref{eq:red_posterior} in a full-space formulation.

\section{Conclusion}
\label{sec:conclusion}

Full-space and augmented-Lagrangian formulations are powerful tools for constrained Bayesian inverse problems, but a hard constraint does not uniquely define a posterior measure. For finite-dimensional discretised state equations \(c(\theta,u)=0\), a naive Gaussian penalty on the residual converges to a zero-noise residual posterior whose parameter marginal is reweighted by \(\abs{\det\D_u c}^{-1}\). This residual posterior is a legitimate target when the residual itself is modelled as a noisy observed quantity. It is not, in general, the reduced Bayesian posterior lifted to the constraint graph.

The graph-lifted reduced posterior is recovered as the hard-constraint limit of a full-space finite-penalty formulation by multiplying by the state-Jacobian determinant \(\abs{\det\D_u c}\), with the additional \(\frac12\log\det R\) term when an unnormalised weighted residual penalty \(c^TRc\) is used. If the residual term is instead included as a normalised Gaussian likelihood, this weight determinant belongs to that likelihood normalisation and should not be counted twice. This correction explains why algebraically equivalent constraints can be optimisation-equivalent but not automatically posterior-equivalent, and why residual-scaling diagnostics are a useful audit of the intended target measure.

The resulting hierarchy is: residual convergence gives feasibility; determinant correction specifies the graph-posterior target; invariant MCMC, valid SMC, or Metropolis-corrected manifold dynamics supply posterior sampling validity; augmented-Lagrangian, ADMM, SVGD, and other particle or splitting methods improve exploration, conditioning, initialisation, or variational approximation but do not by themselves determine the posterior measure. This separation is the main practical message. Full-space and dual-space machinery can be highly effective, provided the target measure is declared first and the sampler is then constructed to preserve or explicitly approximate that target.

\section*{Acknowledgements}
This work was conducted at the Advanced Research Center for Nanolithography, a public-private partnership between the University of Amsterdam (UvA), Vrije Universiteit Amsterdam (VU), Rijksuniversiteit Groningen (RUG), the Netherlands Organization for Scientific Research (NWO), and the semiconductor equipment manufacturer ASML. This work made use of the ARCNL Minerva cluster hosted at HOPSTER at UvA. E.O. is grateful for a WISE Fellowship from the NWO and acknowledges support via Holland High Tech through a public--private partnership in research and development within the Dutch top sector of High-Tech Systems and Materials (HTSM).

\section*{Code and reproducibility}
\label{sec:code_reproducibility}

The numerical diagnostics in Section~\ref{sec:pde_protocol} were generated with \detcorr{} \cite{DetcorrSoftware2026}, a Python package for determinant-corrected constrained posterior calculations. The package is sampler-agnostic: it evaluates \(\log\abs{\det \D_u c}\), weighted-residual volume terms, finite-penalty target addends, and diagnostic records that keep feasibility quantities such as \(\norm{c(\theta,u)}\) separate from calibration quantities such as effective sample size and split-\(\hat R\). The public release is available at \url{https://github.com/Olsson-Materials-Modelling/detcorr} and the archived release is available at \url{10.5281/zenodo.20595686}.

\section*{Author contributions}
\label{sec:author_contributions}

E.O. and J.C. jointly conceived the study, developed the methodological framework, and formulated the determinant-correction argument. J.C. led the software implementation and computational workflow, including the \detcorr{} package interface, benchmark scripts, curation of validation outputs, formal analysis, and visualization. E.O. supervised the project, administered the research programme, secured funding, and guided the overall scientific direction. J.C. prepared the original manuscript draft. E.O. and J.C. jointly validated and interpreted the mathematical and computational results, revised and edited the manuscript, and approved the final version.

\appendix

\section{Surface-density calculation}
\label{app:surface_density}

This appendix verifies \eqref{eq:graph_surface_density}. Let \(\Phi(\theta)=(\theta,G(\theta))\). The Jacobian of \(\Phi\) is
\begin{equation}
 \D\Phi(\theta) =
 \begin{bmatrix}
 I_p\\\D G(\theta)
 \end{bmatrix},
\end{equation}
so
\begin{equation}
 \dd\sigma_\Gamma
 =
 \sdet{\D\Phi(\theta)^T\D\Phi(\theta)}\dd\theta
 =
 \sdet{I_p+\D G(\theta)^T\D G(\theta)}\dd\theta.
\end{equation}
The graph-lifted posterior is
\begin{equation}
 \pi_\Gamma(\Phi(\theta))\dd\sigma_\Gamma
 =
 \frac{1}{Z_{\rm red}}r(\theta,G(\theta))\dd\theta.
\end{equation}
Therefore
\begin{equation}
 \frac{\dd\pi_\Gamma}{\dd\sigma_\Gamma}(\theta,G(\theta))
 =
 \frac{1}{Z_{\rm red}}
 \frac{r(\theta,G(\theta))}{\sdet{I_p+\D G(\theta)^T\D G(\theta)}}.
 \label{eq:graph_surface_appendix_1}
\end{equation}
Since \(\D_\theta c + \D_u c\D G=0\) on \(\Gamma\),
\begin{equation}
 J_cJ_c^T
 =
 \D_\theta c\D_\theta c^T + \D_u c\D_u c^T
 =
 \D_u c\left(I_q + \D G\D G^T\right)\D_u c^T.
\end{equation}
Thus
\begin{equation}
 \sdet{J_cJ_c^T}
 =
 \abs{\det\D_u c}\,
 \sdet{I_q+\D G\D G^T}.
\end{equation}
The nonzero eigenvalues of \(\D G\D G^T\) and \(\D G^T\D G\) coincide, so
\begin{equation}
 \det(I_q+\D G\D G^T)=\det(I_p+\D G^T\D G).
\end{equation}
Substituting this identity into \eqref{eq:graph_surface_appendix_1} gives \eqref{eq:graph_surface_density}.

\section{Software interface for corrected finite-penalty targets}
\label{app:detcorr_api}

This appendix records the minimal \detcorr{} interface used by the validation benchmarks. The package does not choose a posterior measure and does not implement an exact sampler by itself. Instead, it exposes the state-Jacobian correction and packages the resulting log-density terms so that an external MCMC, SMC, manifold, or particle method can use them in its own proposal, weighting, acceptance, or approximation machinery. All sampler-facing helpers require an explicit declaration of the intended target measure.

For a residual function \(c(\theta,u)\), the graph-target finite-penalty log-density contribution is
\begin{equation}
 \log\abs{\det \D_u c(\theta,u)}
 -\frac{\rho}{2}\norm{c(\theta,u)}^2,
 \label{eq:app_unweighted_target_helper}
\end{equation}
plus the user's prior and likelihood log density. The direct helper evaluates both the residual and the determinant correction:
\begin{verbatim}
import detcorr

def c(theta, u):
 return residual_vector(theta, u)

def base(theta, u):
 return log_prior(theta) + log_likelihood(theta, u)

evaluation = detcorr.evaluate_corrected_finite_penalty_target(
 c,
 theta,
 u,
 rho=100.0,
 target_measure="graph",
 base_logdensity=base(theta, u),
 log_correction_kwargs={"backend": "scipy", "singular_policy": "raise"},
)
logpi = evaluation.logdensity
diagnostics = evaluation.sampler_diagnostics()
\end{verbatim}
When the state Jacobian is already available from a PDE code or automatic differentiation, it can be supplied explicitly:
\begin{verbatim}
J_u = Du_c(theta, u)
logdet = detcorr.log_correction(jacobian=J_u, method="sparse_lu")

evaluation = detcorr.corrected_finite_penalty_target(
  c(theta, u),
  rho=100.0,
  logdet=logdet,
  target_measure="graph",
  base_logdensity=base(theta, u),
)
\end{verbatim}
The determinant route can therefore be exact dense linear algebra, sparse LU, Cholesky for symmetric positive definite operators, a structured diagonal or triangular product, an autodiff Jacobian, or a matrix-free estimator with recorded estimator diagnostics, using stochastic trace and stochastic Lanczos ideas where exact factorization is unavailable \cite{Hutchinson1989,UbaruChenSaad2017,MeyerMuscoMuscoWoodruff2021}.

For a weighted residual penalty \(\frac{\rho}{2}c(\theta,u)^T R(\theta)c(\theta,u)\), where \(R(\theta)\) is symmetric positive definite, the software helper implements Corollary~\ref{cor:weighted_residuals}. The corresponding graph-target log-density contribution is
\begin{equation}
 \log\abs{\det \D_u c(\theta,u)}
 + \frac{1}{2}\log\det R(\theta)
 - \frac{\rho}{2} c(\theta,u)^T R(\theta)c(\theta,u).
 \label{eq:app_weighted_target_helper}
\end{equation}
Equation~\eqref{eq:app_weighted_target_helper} is therefore the software form of the main-text weighted-residual correction, not a separate derivation. The weight can be passed as \texttt{R=R(theta)}, as a Cholesky factor, as a preconditioner, or as a precomputed weight log-determinant. Singular, indefinite, or negative-determinant weights are outside Corollary~\ref{cor:weighted_residuals} and should be rejected or handled by an explicitly different model. 

Declaring \texttt{target\_measure="graph"} includes the state-Jacobian correction, and, when a weighted residual is supplied, the \(\frac12\log\det R\) volume term. Declaring \texttt{"residual\_noise"} intentionally omits the graph correction and rejects an accidental state-Jacobian log determinant. This explicit declaration is meant to prevent the software from silently switching between the graph posterior and the residual-noise posterior.

The benchmark scripts use only this interface. The scalar residual-scaling test passes \(c\), \(A(\theta)c\), and the compensating weight \(R_A(\theta)=A(\theta)^{-T}A(\theta)^{-1}\), verifying the invariance in Remark~\ref{rem:weighted_rescaling_invariance}. The elliptic benchmark passes the affine state Jacobian \(A(m)\) and evaluates \(\log\abs{\det A(m)}\) by an exact SPD route. Feasibility diagnostics report residual norms and projection errors; calibration diagnostics report effective sample size, split-\(\hat R\), and simulation-based calibration hooks. The two diagnostic types are intentionally kept separate because residual convergence is necessary for a hard-constraint limit but not sufficient for posterior calibration.

\section{Multiple state solutions}
\label{app:multiple_roots}

If the state equation has finitely many nonsingular solutions \(u=G_j(\theta)\), \(j=1,\ldots,J(\theta)\), the same change-of-variables argument gives the naive penalty limit
\begin{equation}
 \pi^\theta_{\rm res}(\theta)
 \propto
 \sum_{j=1}^{J(\theta)}
 r(\theta,G_j(\theta))
 \abs{\det\D_u c(\theta,G_j(\theta))}^{-1}.
\end{equation}
The determinant-corrected penalty gives
\begin{equation}
 \widetilde\pi^\theta(\theta)
 \propto
 \sum_{j=1}^{J(\theta)}
 r(\theta,G_j(\theta)).
\end{equation}
Which of these is appropriate depends on the intended model over branches. The single-solution case in the main text is the common setting for well-posed discretised state equations.

\bibliographystyle{unsrturl}
\bibliography{references}

@article{Tarantola1984,
  author  = {Albert Tarantola},
  title   = {Inversion of seismic reflection data in the acoustic approximation},
  journal = {Geophysics},
  volume  = {49},
  number  = {8},
  pages   = {1259--1266},
  year    = {1984},
  doi     = {10.1190/1.1441754}
}

@book{Tarantola2005,
  author    = {Albert Tarantola},
  title     = {Inverse Problem Theory and Methods for Model Parameter Estimation},
  publisher = {SIAM},
  year      = {2005},
  doi       = {10.1137/1.9780898717921}
}

@book{KaipioSomersalo2005,
  author    = {Jari Kaipio and Erkki Somersalo},
  title     = {Statistical and Computational Inverse Problems},
  publisher = {Springer},
  year      = {2005},
  doi       = {10.1007/b138659}
}

@article{Stuart2010,
  author  = {Andrew M. Stuart},
  title   = {Inverse problems: a {Bayesian} perspective},
  journal = {Acta Numerica},
  volume  = {19},
  pages   = {451--559},
  year    = {2010},
  doi     = {10.1017/S0962492910000061}
}

@article{CalvettiKaipioSomersalo2014,
  author  = {Daniela Calvetti and Jari Kaipio and Erkki Somersalo},
  title   = {Inverse problems in the {Bayesian} framework},
  journal = {Inverse Problems},
  volume  = {30},
  number  = {11},
  pages   = {110301},
  year    = {2014},
  doi     = {10.1088/0266-5611/30/11/110301}
}

@incollection{DashtiStuart2017,
  author    = {Masoumeh Dashti and Andrew M. Stuart},
  title     = {The {Bayesian} approach to inverse problems},
  booktitle = {Handbook of Uncertainty Quantification},
  pages     = {311--428},
  publisher = {Springer},
  year      = {2017},
  doi       = {10.1007/978-3-319-12385-1_7}
}

@article{VirieuxOperto2009,
  author  = {Jean Virieux and St{\'e}phane Operto},
  title   = {An overview of full-waveform inversion in exploration geophysics},
  journal = {Geophysics},
  volume  = {74},
  number  = {6},
  pages   = {WCC1--WCC26},
  year    = {2009},
  doi     = {10.1190/1.3238367}
}

@article{BuiThanhGirolami2014,
  author  = {Tan Bui-Thanh and Mark Girolami},
  title   = {Solving large-scale {PDE}-constrained {Bayesian} inverse problems with {Riemann} manifold {Hamiltonian} {Monte} {Carlo}},
  journal = {Inverse Problems},
  volume  = {30},
  number  = {11},
  pages   = {114014},
  year    = {2014},
  doi     = {10.1088/0266-5611/30/11/114014}
}

@article{BuiThanhGhattasMartinStadler2013,
  author  = {Tan Bui-Thanh and Omar Ghattas and James Martin and Georg Stadler},
  title   = {A computational framework for infinite-dimensional {Bayesian} inverse problems. {Part I}: The linearized case, with application to global seismic inversion},
  journal = {SIAM Journal on Scientific Computing},
  volume  = {35},
  number  = {6},
  pages   = {A2494--A2523},
  year    = {2013},
  doi     = {10.1137/12089586X}
}

@article{PetraMartinStadlerGhattas2014,
  author  = {Noemi Petra and James Martin and Georg Stadler and Omar Ghattas},
  title   = {A computational framework for infinite-dimensional {Bayesian} inverse problems. {Part II}: Stochastic {Newton} {MCMC} with application to ice sheet flow inverse problems},
  journal = {SIAM Journal on Scientific Computing},
  volume  = {36},
  number  = {4},
  pages   = {A1525--A1555},
  year    = {2014},
  doi     = {10.1137/130934805}
}

@article{BuiThanhNguyen2016,
  author  = {Tan Bui-Thanh and Quoc P. Nguyen},
  title   = {{FEM}-based discretization-invariant {MCMC} methods for {PDE}-constrained {Bayesian} inverse problems},
  journal = {Inverse Problems and Imaging},
  volume  = {10},
  number  = {4},
  pages   = {943--975},
  year    = {2016},
  doi     = {10.3934/ipi.2016028}
}

@article{FlathWilcoxAkcelikHillVanBloemenWaandersGhattas2011,
  author  = {H. P. Flath and L. C. Wilcox and V. Akcelik and J. Hill and B. van Bloemen Waanders and O. Ghattas},
  title   = {Fast algorithms for {Bayesian} uncertainty quantification in large-scale linear inverse problems based on low-rank partial Hessian approximations},
  journal = {SIAM Journal on Scientific Computing},
  volume  = {33},
  number  = {1},
  pages   = {407--432},
  year    = {2011},
  doi     = {10.1137/090780717}
}

@article{MartinWilcoxBursteddeGhattas2012,
  author  = {James Martin and Lucas C. Wilcox and Carsten Burstedde and Omar Ghattas},
  title   = {A stochastic {Newton} {MCMC} method for large-scale statistical inverse problems with application to seismic inversion},
  journal = {SIAM Journal on Scientific Computing},
  volume  = {34},
  number  = {3},
  pages   = {A1460--A1487},
  year    = {2012},
  doi     = {10.1137/110845598}
}

@article{CuiMartinMarzoukSolonenSpantini2014,
  author  = {Tiangang Cui and James Martin and Youssef M. Marzouk and Antti Solonen and Alessio Spantini},
  title   = {Likelihood-informed dimension reduction for nonlinear inverse problems},
  journal = {Inverse Problems},
  volume  = {30},
  number  = {11},
  pages   = {114015},
  year    = {2014},
  doi     = {10.1088/0266-5611/30/11/114015}
}

@article{CuiLawMarzouk2016,
  author  = {Tiangang Cui and Kody J. H. Law and Youssef M. Marzouk},
  title   = {Dimension-independent likelihood-informed {MCMC}},
  journal = {Journal of Computational Physics},
  volume  = {304},
  pages   = {109--137},
  year    = {2016},
  doi     = {10.1016/j.jcp.2015.10.008}
}

@article{HaberAscher2001,
  author  = {Eldad Haber and Uri M. Ascher},
  title   = {Preconditioned all-at-once methods for large, sparse parameter estimation problems},
  journal = {Inverse Problems},
  volume  = {17},
  number  = {6},
  pages   = {1847--1864},
  year    = {2001},
  doi     = {10.1088/0266-5611/17/6/319}
}

@book{HinzePinnauUlbrichUlbrich2009,
  author    = {Michael Hinze and Ren{\'e} Pinnau and Michael Ulbrich and Stefan Ulbrich},
  title     = {Optimization with {PDE} Constraints},
  publisher = {Springer},
  year      = {2009},
  doi       = {10.1007/978-1-4020-8839-1}
}

@book{BorziSchulz2012,
  author    = {Alfio Borz{\`i} and Volker Schulz},
  title     = {Computational Optimization of Systems Governed by Partial Differential Equations},
  publisher = {SIAM},
  year      = {2012},
  doi       = {10.1137/1.9781611972054}
}

@article{BenziGolubLiesen2005,
  author  = {Michele Benzi and Gene H. Golub and J{\"o}rg Liesen},
  title   = {Numerical solution of saddle point problems},
  journal = {Acta Numerica},
  volume  = {14},
  pages   = {1--137},
  year    = {2005},
  doi     = {10.1017/S0962492904000212}
}

@article{vanLeeuwenHerrmann2016,
  author  = {Tristan van Leeuwen and Felix J. Herrmann},
  title   = {A penalty method for {PDE}-constrained optimization in inverse problems},
  journal = {Inverse Problems},
  volume  = {32},
  number  = {1},
  pages   = {015007},
  year    = {2016},
  doi     = {10.1088/0266-5611/32/1/015007}
}

@article{CotterRobertsStuartWhite2013,
  author  = {Simon L. Cotter and Gareth O. Roberts and Andrew M. Stuart and David White},
  title   = {{MCMC} methods for functions: modifying old algorithms to make them faster},
  journal = {Statistical Science},
  volume  = {28},
  number  = {3},
  pages   = {424--446},
  year    = {2013},
  doi     = {10.1214/13-STS421}
}

@article{BeskosPinskiSanzSernaStuart2011,
  author  = {Alexandros Beskos and Frank J. Pinski and J. M. Sanz-Serna and Andrew M. Stuart},
  title   = {Hybrid {Monte} {Carlo} on {Hilbert} spaces},
  journal = {Stochastic Processes and their Applications},
  volume  = {121},
  number  = {10},
  pages   = {2201--2230},
  year    = {2011},
  doi     = {10.1016/j.spa.2011.06.003}
}

@article{HairerStuartVollmer2014,
  author  = {Martin Hairer and Andrew M. Stuart and Sebastian J. Vollmer},
  title   = {Spectral gaps for a {Metropolis}--{Hastings} algorithm in infinite dimensions},
  journal = {Annals of Applied Probability},
  volume  = {24},
  number  = {6},
  pages   = {2455--2490},
  year    = {2014},
  doi     = {10.1214/13-AAP982}
}

@article{Law2014,
  author  = {Kody J. H. Law},
  title   = {Proposals which speed up function-space {MCMC}},
  journal = {Journal of Computational and Applied Mathematics},
  volume  = {262},
  pages   = {127--138},
  year    = {2014},
  doi     = {10.1016/j.cam.2013.07.026}
}

@article{GirolamiCalderhead2011,
  author  = {Mark Girolami and Ben Calderhead},
  title   = {{Riemann} manifold {Langevin} and {Hamiltonian} {Monte} {Carlo} methods},
  journal = {Journal of the Royal Statistical Society: Series B},
  volume  = {73},
  number  = {2},
  pages   = {123--214},
  year    = {2011},
  doi     = {10.1111/j.1467-9868.2010.00765.x}
}

@article{BeskosGirolamiLanFarrellStuart2017,
  author  = {Alexandros Beskos and Mark Girolami and Shiwei Lan and Patrick E. Farrell and Andrew M. Stuart},
  title   = {Geometric {MCMC} for infinite-dimensional inverse problems},
  journal = {Journal of Computational Physics},
  volume  = {335},
  pages   = {327--351},
  year    = {2017},
  doi     = {10.1016/j.jcp.2016.12.041}
}

@article{EcksteinBertsekas1992,
  author  = {Jonathan Eckstein and Dimitri P. Bertsekas},
  title   = {On the {Douglas}--{Rachford} splitting method and the proximal point algorithm for maximal monotone operators},
  journal = {Mathematical Programming},
  volume  = {55},
  pages   = {293--318},
  year    = {1992},
  doi     = {10.1007/BF01581204}
}

@article{BoydParikhChuPeleatoEckstein2011,
  author  = {Stephen Boyd and Neal Parikh and Eric Chu and Borja Peleato and Jonathan Eckstein},
  title   = {Distributed optimization and statistical learning via the alternating direction method of multipliers},
  journal = {Foundations and Trends in Machine Learning},
  volume  = {3},
  number  = {1},
  pages   = {1--122},
  year    = {2011},
  doi     = {10.1561/2200000016}
}

@article{ParikhBoyd2014,
  author  = {Neal Parikh and Stephen Boyd},
  title   = {Proximal algorithms},
  journal = {Foundations and Trends in Optimization},
  volume  = {1},
  number  = {3},
  pages   = {127--239},
  year    = {2014},
  doi     = {10.1561/2400000003}
}

@article{VonoPaulinDoucet2022,
  author  = {Maxime Vono and Daniel Paulin and Arnaud Doucet},
  title   = {Efficient {MCMC} sampling with dimension-free convergence rate using {ADMM}-type splitting},
  journal = {Journal of Machine Learning Research},
  volume  = {23},
  number  = {25},
  pages   = {1--69},
  year    = {2022},
  url     = {https://jmlr.org/papers/v23/20-357.html}
}

@inproceedings{LiuWang2016,
  author    = {Qiang Liu and Dilin Wang},
  title     = {{Stein} variational gradient descent: a general purpose {Bayesian} inference algorithm},
  booktitle = {Advances in Neural Information Processing Systems},
  volume    = {29},
  pages     = {2370--2378},
  year      = {2016},
  url       = {https://proceedings.neurips.cc/paper/2016/hash/b3ba8f1bee1238a2f37603d90b58898d-Abstract.html}
}

@article{JiaLiMeng2022,
  author  = {Junxiong Jia and Peijun Li and Deyu Meng},
  title   = {{Stein} variational gradient descent on infinite-dimensional space and applications to statistical inverse problems},
  journal = {SIAM Journal on Numerical Analysis},
  volume  = {60},
  number  = {4},
  pages   = {2225--2252},
  year    = {2022},
  doi     = {10.1137/21M1440773}
}

@misc{SiahkoohiAghazadeGholami2026,
  author        = {Ali Siahkoohi and Kamal Aghazade and Ali Gholami},
  title         = {Dual-space posterior sampling for {Bayesian} inference in constrained inverse problems},
  year          = {2026},
  archivePrefix = {arXiv},
  eprint        = {2603.00393},
  doi           = {10.48550/arXiv.2603.00393},
  note          = {Preprint arXiv:2603.00393}
}

@misc{AghazadeSiahkoohiGholami2026,
  author        = {Kamal Aghazade and Ali Siahkoohi and Ali Gholami},
  title         = {Scalable {Bayesian} full waveform inversion via dual augmented {Lagrangian} {SVGD}},
  year          = {2026},
  archivePrefix = {arXiv},
  eprint        = {2603.24751},
  doi           = {10.48550/arXiv.2603.24751},
  note          = {Preprint arXiv:2603.24751}
}

@misc{XuZengPaisleyZhao2026,
  author        = {Jian Xu and Delu Zeng and John Paisley and Qibin Zhao},
  title         = {The right measure for physics-constrained generation: a co-area correction for posterior-consistent {PDE} inverse problems},
  year          = {2026},
  doi           = {10.48550/arXiv.2606.04804}
}

@book{EvansGariepy2015,
  author    = {Lawrence C. Evans and Ronald F. Gariepy},
  title     = {Measure Theory and Fine Properties of Functions},
  publisher = {CRC Press},
  edition   = {Revised},
  year      = {2015},
  doi       = {10.1201/b18333}
}

@book{Federer1969,
  author    = {Herbert Federer},
  title     = {Geometric Measure Theory},
  series    = {Grundlehren der mathematischen Wissenschaften},
  volume    = {153},
  publisher = {Springer},
  year      = {1969},
  doi       = {10.1007/978-3-642-62010-2},
  note      = {Reprinted in the Classics in Mathematics series, Springer, 1996}
}

@article{ChangPollard1997,
  author  = {Joseph T. Chang and David Pollard},
  title   = {Conditioning as disintegration},
  journal = {Statistical Science},
  volume  = {51},
  number  = {3},
  pages   = {287--317},
  year    = {1997},
  doi     = {10.1111/1467-9574.00056}
}

@article{Fixman1974,
  author  = {Marshall Fixman},
  title   = {Classical statistical mechanics of constraints: a theorem and application to polymers},
  journal = {Proceedings of the National Academy of Sciences of the United States of America},
  volume  = {71},
  number  = {8},
  pages   = {3050--3053},
  year    = {1974},
  doi     = {10.1073/pnas.71.8.3050}
}

@book{Lee2013,
  author    = {John M. Lee},
  title     = {Introduction to Smooth Manifolds},
  publisher = {Springer},
  edition   = {Second},
  series    = {Graduate Texts in Mathematics},
  volume    = {218},
  year      = {2013},
  doi       = {10.1007/978-1-4419-9982-5}
}

@article{MetropolisRosenbluthTellerTeller1953,
  author  = {Nicholas Metropolis and Arianna W. Rosenbluth and Marshall N. Rosenbluth and Augusta H. Teller and Edward Teller},
  title   = {Equation of state calculations by fast computing machines},
  journal = {Journal of Chemical Physics},
  volume  = {21},
  number  = {6},
  pages   = {1087--1092},
  year    = {1953},
  doi     = {10.1063/1.1699114}
}

@article{Hastings1970,
  author  = {W. Keith Hastings},
  title   = {{Monte} {Carlo} sampling methods using {Markov} chains and their applications},
  journal = {Biometrika},
  volume  = {57},
  number  = {1},
  pages   = {97--109},
  year    = {1970},
  doi     = {10.1093/biomet/57.1.97}
}

@article{DuaneKennedyPendletonRoweth1987,
  author  = {Simon Duane and Anthony D. Kennedy and Brian J. Pendleton and Duncan Roweth},
  title   = {Hybrid {Monte} {Carlo}},
  journal = {Physics Letters B},
  volume  = {195},
  number  = {2},
  pages   = {216--222},
  year    = {1987},
  doi     = {10.1016/0370-2693(87)91197-X}
}

@incollection{Neal2011,
  author    = {Radford M. Neal},
  title     = {{MCMC} using {Hamiltonian} dynamics},
  booktitle = {Handbook of Markov Chain Monte Carlo},
  editor    = {Steve Brooks and Andrew Gelman and Galin Jones and Xiao-Li Meng},
  pages     = {113--162},
  publisher = {CRC Press},
  year      = {2011},
  doi       = {10.1201/b10905-6}
}

@article{Hestenes1969,
  author  = {Magnus R. Hestenes},
  title   = {Multiplier and gradient methods},
  journal = {Journal of Optimization Theory and Applications},
  volume  = {4},
  pages   = {303--320},
  year    = {1969},
  doi     = {10.1007/BF00927673}
}

@article{Rockafellar1974,
  author  = {R. Tyrrell Rockafellar},
  title   = {Augmented {Lagrange} multiplier functions and duality in nonconvex programming},
  journal = {SIAM Journal on Control},
  volume  = {12},
  number  = {2},
  pages   = {268--285},
  year    = {1974},
  doi     = {10.1137/0312021}
}

@article{ChambollePock2011,
  author  = {Antonin Chambolle and Thomas Pock},
  title   = {A first-order primal-dual algorithm for convex problems with applications to imaging},
  journal = {Journal of Mathematical Imaging and Vision},
  volume  = {40},
  pages   = {120--145},
  year    = {2011},
  doi     = {10.1007/s10851-010-0251-1}
}

@article{DelMoralDoucetJasra2006,
  author  = {Pierre Del Moral and Arnaud Doucet and Ajay Jasra},
  title   = {Sequential {Monte} {Carlo} samplers},
  journal = {Journal of the Royal Statistical Society: Series B},
  volume  = {68},
  number  = {3},
  pages   = {411--436},
  year    = {2006},
  doi     = {10.1111/j.1467-9868.2006.00553.x}
}

@article{Chopin2002,
  author  = {Nicolas Chopin},
  title   = {A sequential particle filter method for static models},
  journal = {Biometrika},
  volume  = {89},
  number  = {3},
  pages   = {539--552},
  year    = {2002},
  doi     = {10.1093/biomet/89.3.539}
}

@article{Neal2001,
  author  = {Radford M. Neal},
  title   = {Annealed importance sampling},
  journal = {Statistics and Computing},
  volume  = {11},
  pages   = {125--139},
  year    = {2001},
  doi     = {10.1023/A:1008923215028}
}

@article{BeskosJasraMuzafferStuart2015,
  author  = {Alexandros Beskos and Ajay Jasra and Ege A. Muzaffer and Andrew M. Stuart},
  title   = {Sequential {Monte} {Carlo} methods for {Bayesian} elliptic inverse problems},
  journal = {Statistics and Computing},
  volume  = {25},
  number  = {4},
  pages   = {727--737},
  year    = {2015},
  doi     = {10.1007/s11222-015-9556-7}
}

@incollection{DiaconisHolmesShahshahani2013,
  author    = {Persi Diaconis and Susan Holmes and Mehrdad Shahshahani},
  title     = {Sampling from a manifold},
  booktitle = {Advances in Modern Statistical Theory and Applications: A Festschrift in Honor of Morris L. Eaton},
  series    = {Institute of Mathematical Statistics Collections},
  volume    = {10},
  pages     = {102--125},
  publisher = {Institute of Mathematical Statistics},
  year      = {2013},
  doi       = {10.1214/12-IMSCOLL1006}
}

@article{ByrneGirolami2013,
  author  = {Simon Byrne and Mark Girolami},
  title   = {Geodesic {Monte} {Carlo} on embedded manifolds},
  journal = {Scandinavian Journal of Statistics},
  volume  = {40},
  number  = {4},
  pages   = {825--845},
  year    = {2013},
  doi     = {10.1111/sjos.12036}
}

@article{ZappaHolmesCerfonGoodman2018,
  author  = {Emilio Zappa and Miranda Holmes-Cerfon and Jonathan Goodman},
  title   = {{Monte} {Carlo} on manifolds: sampling densities and integrating functions},
  journal = {Communications on Pure and Applied Mathematics},
  volume  = {71},
  number  = {12},
  pages   = {2609--2647},
  year    = {2018},
  doi     = {10.1002/cpa.21783}
}

@article{LelievreRoussetStoltz2019,
  author  = {Tony Leli{\`e}vre and Mathias Rousset and Gabriel Stoltz},
  title   = {Hybrid {Monte} {Carlo} methods for sampling probability measures on submanifolds},
  journal = {Numerische Mathematik},
  volume  = {143},
  number  = {2},
  pages   = {379--421},
  year    = {2019},
  doi     = {10.1007/s00211-019-01056-4}
}

@article{GrahamThieryBeskos2022,
  author  = {Matthew M. Graham and Alexandre H. Thiery and Alexandros Beskos},
  title   = {Manifold {Markov} chain {Monte} {Carlo} methods for {Bayesian} inference in diffusion models},
  journal = {Journal of the Royal Statistical Society: Series B},
  volume  = {84},
  number  = {4},
  pages   = {1229--1256},
  year    = {2022},
  doi     = {10.1111/rssb.12497}
}

@article{GelmanRubin1992,
  author  = {Andrew Gelman and Donald B. Rubin},
  title   = {Inference from iterative simulation using multiple sequences},
  journal = {Statistical Science},
  volume  = {7},
  number  = {4},
  pages   = {457--472},
  year    = {1992},
  doi     = {10.1214/ss/1177011136}
}

@article{VehtariGelmanSimpsonCarpenterBurkner2021,
  author  = {Aki Vehtari and Andrew Gelman and Daniel Simpson and Bob Carpenter and Paul-Christian B{\"u}rkner},
  title   = {Rank-normalization, folding, and localization: an improved \(\widehat R\) for assessing convergence of {MCMC}},
  journal = {Bayesian Analysis},
  volume  = {16},
  number  = {2},
  pages   = {667--718},
  year    = {2021},
  doi     = {10.1214/20-BA1221}
}

@article{CookGelmanRubin2006,
  author  = {Samantha R. Cook and Andrew Gelman and Donald B. Rubin},
  title   = {Validation of software for {Bayesian} models using posterior quantiles},
  journal = {Journal of Computational and Graphical Statistics},
  volume  = {15},
  number  = {3},
  pages   = {675--692},
  year    = {2006},
  doi     = {10.1198/106186006X136976}
}

@article{ModrakMoonKimBurknerHuurreFaltejskovaGelmanVehtari2025,
  author  = {Martin Modr{\'a}k and Angie H. Moon and Shinyoung Kim and Paul-Christian B{\"u}rkner and Niko Huurre and Kate{\v r}ina Faltejskov{\'a} and Andrew Gelman and Aki Vehtari},
  title   = {Simulation-based calibration checking for {Bayesian} computation: the choice of test quantities shapes sensitivity},
  journal = {Bayesian Analysis},
  volume  = {20},
  number  = {2},
  pages   = {461--488},
  year    = {2025},
  doi     = {10.1214/23-BA1404}
}

@article{KennedyOHagan2001,
  author  = {Marc C. Kennedy and Anthony O'Hagan},
  title   = {{Bayesian} calibration of computer models},
  journal = {Journal of the Royal Statistical Society: Series B},
  volume  = {63},
  number  = {3},
  pages   = {425--464},
  year    = {2001},
  doi     = {10.1111/1467-9868.00294}
}

@article{Wilkinson2013,
  author  = {Richard D. Wilkinson},
  title   = {Approximate {Bayesian} computation ({ABC}) gives exact results under the assumption of model error},
  journal = {Statistical Applications in Genetics and Molecular Biology},
  volume  = {12},
  number  = {2},
  pages   = {129--141},
  year    = {2013},
  doi     = {10.1515/sagmb-2013-0010}
}

@article{RyckaertCiccottiBerendsen1977,
  author  = {Jean-Paul Ryckaert and Giovanni Ciccotti and Herman J. C. Berendsen},
  title   = {Numerical integration of the {Cartesian} equations of motion of a system with constraints: molecular dynamics of n-alkanes},
  journal = {Journal of Computational Physics},
  volume  = {23},
  number  = {3},
  pages   = {327--341},
  year    = {1977},
  doi     = {10.1016/0021-9991(77)90098-5}
}

@article{Andersen1983,
  author  = {Hans C. Andersen},
  title   = {{RATTLE}: A ``velocity'' version of the {SHAKE} algorithm for molecular dynamics calculations},
  journal = {Journal of Computational Physics},
  volume  = {52},
  number  = {1},
  pages   = {24--34},
  year    = {1983},
  doi     = {10.1016/0021-9991(83)90014-1}
}

@book{HairerLubichWanner2006,
  author    = {Ernst Hairer and Christian Lubich and Gerhard Wanner},
  title     = {Geometric Numerical Integration: Structure-Preserving Algorithms for Ordinary Differential Equations},
  publisher = {Springer},
  edition   = {Second},
  series    = {Springer Series in Computational Mathematics},
  volume    = {31},
  year      = {2006},
  doi       = {10.1007/3-540-30666-8}
}

@article{Hutchinson1989,
  author  = {Michael F. Hutchinson},
  title   = {A stochastic estimator of the trace of the influence matrix for {Laplacian} smoothing splines},
  journal = {Communications in Statistics -- Simulation and Computation},
  volume  = {18},
  number  = {3},
  pages   = {1059--1076},
  year    = {1989},
  doi     = {10.1080/03610918908812806}
}

@article{UbaruChenSaad2017,
  author  = {Shashanka Ubaru and Jie Chen and Yousef Saad},
  title   = {Fast estimation of {$\mathrm{tr}(f(A))$} via stochastic {Lanczos} quadrature},
  journal = {SIAM Journal on Matrix Analysis and Applications},
  volume  = {38},
  number  = {4},
  pages   = {1075--1099},
  year    = {2017},
  doi     = {10.1137/16M1104974}
}

@inproceedings{MeyerMuscoMuscoWoodruff2021,
  author    = {Raphael A. Meyer and Cameron Musco and Christopher Musco and David P. Woodruff},
  title     = {{Hutch++}: optimal stochastic trace estimation},
  booktitle = {Proceedings of the 2021 Symposium on Simplicity in Algorithms},
  pages     = {142--155},
  publisher = {SIAM},
  year      = {2021},
  doi       = {10.1137/1.9781611976496.16}
}

@misc{DetcorrSoftware2026,
    author = {{Jonathon Cottom and Emilia Olsson}},
    title        = {{detcorr}: determinant corrections for constrained {Bayesian} inverse problems},
    howpublished = {Python package, version 0.10.0},
    year         = {2026},doi          = {10.5281/zenodo.20595686},
    url          = {https://github.com/Olsson-Materials-Modelling/detcorr},
   }

\end{document}